\documentclass[11pt,a4paper]{article}
\usepackage[left=2cm, top=2.5cm,bottom=2.5cm,right=2cm]{geometry}
\usepackage{mathtools,amssymb,amsthm,mathrsfs,calc,graphicx,xcolor,cleveref,dsfont,tikz,pgfplots,bm,url,tabularx}
\usepackage[british]{babel}
\usepackage{amsfonts}             
\usepackage[T1]{fontenc}
\usepackage{enumitem}
\usepackage[color=green]{todonotes}
\usepackage{tikz-cd}
\usetikzlibrary{calc} 
\usetikzlibrary{intersections}
\usepackage[labelfont=bf]{caption}
\usepackage[font=small]{subcaption}

\theoremstyle{plain}
\newtheorem{theorem}{Theorem}
\newtheorem{lemma}[theorem]{Lemma}
\newtheorem{corollary}[theorem]{Corollary}
\newtheorem{proposition}[theorem]{Proposition}

\theoremstyle{definition}

\theoremstyle{remark}
\newtheorem{remark}[theorem]{Remark}

\newcommand{\cov}{\operatorname{Cov}}

\newcommand{\vol}{\mathop{\mathrm{vol}}\nolimits}

\def\CC{\mathbb{C}}

\def\EE{\mathbb{E}}

\def\HH{\mathbb{H}}

\def\NN{\mathbb{N}}

\def\PP{\mathbb{P}}

\def\RR{\mathbb{R}}
\def\SS{\mathbb{S}}

\def\cA{\mathcal{A}}

\def\cH{\mathcal{H}}

\newcommand{\Res}{\mathop{\mathrm{Res}}\nolimits}

\newcommand{\eps}{\varepsilon}

\newcommand{\E}{\mathbb{E}}
\newcommand{\Var}{\mathrm{Var}}

\newcommand{\arcosh}{\mathrm{arcosh\,}}

\DeclareMathOperator{\re}{Re}

\makeatletter
\let\@fnsymbol\@alph
\makeatother

\begin{document}
	
	\title{\bfseries The volume of hyperbolic Poisson zero cells:\\ critical divergence and exact second moment}
	
	\author{Tillmann B\"uhler\footnotemark[1]\;\; and Christoph Th\"ale\footnotemark[2]}
	
	\date{}
	\renewcommand{\thefootnote}{\fnsymbol{footnote}}
    \footnotetext[1]{Karlsruhe Institute of Technology, Germany. Email: tillmann.buehler@kit.edu}
	\footnotetext[2]{Ruhr University Bochum, Germany. Email: christoph.thaele@rub.de}
	
	\maketitle
	
	\begin{abstract}
		\noindent
		We investigate the second volume moment of the zero cell $Z_o$ of a Poisson hyperplane tessellation with intensity $\gamma$ in the $d$-dimensional hyperbolic space. We focus on the phase transition at the critical intensity $\gamma_c^{(d)}$, the minimum value for which $Z_o$ is almost surely bounded. In the critical regime $\gamma=\gamma_c^{(d)}$, we show that the second volume moment of the restricted zero cell $Z_o \cap B_R$, where $B_R$ is a hyperbolic ball of radius $R$ centred at $o$, diverges in any dimension at the universal rate $R^3$ as $R \to \infty$. In the supercritical case $\gamma > \gamma_c^{(d)}$, we prove that the full second volume moment is finite. Using tools from harmonic analysis in hyperbolic space, we derive an exact expression for this moment in terms of the Meijer $G$-function. Furthermore, we determine the asymptotic behaviour of the second moment as $\gamma \to \infty$ and as $\gamma \downarrow \gamma_c^{(d)}$, facilitating a direct comparison with the corresponding Euclidean values as well as the mean-field universality class of percolation theory.
		\\
		
	\end{abstract}

\section{Introduction and main results}

Poisson hyperplane tessellations are among the classical models in stochastic geometry. They are generated by a Poisson process on the space of hyperplanes in $\mathbb{R}^d$, whose atoms subdivide the ambient space into a countable collection of random convex polytopes, called cells. A prominent object of study in this context is the so-called zero cell $Z_o$, defined as the cell containing the origin. The geometric properties of $Z_o$, such as its number of $k$-faces, surface area, and volume, have been extensively studied since the early days of stochastic geometry. In particular, the first moments of various geometric functionals associated with $Z_o$ can be computed explicitly. For a stationary and isotropic Poisson hyperplane tessellation in $\RR^d$ with intensity $\gamma>0$, an explicit formula for the second moment of the volume is also available. Namely, if $\vol_d(Z_o)$ denotes the $d$-dimensional volume of $Z_o$, we have
\begin{equation}\label{eq:SecondMomentEuclidean}
    \EE\vol_d(Z_o)^2 = \frac{2^{4d} \pi^{\frac{4d-3}{2}}}{\gamma^{2d}} \frac{\Gamma\left(d+\frac{3}{2}\right)}{\Gamma\left(\frac{3d+3}{2}\right)} \frac{\Gamma\left(\frac{d+1}{2}\right)^{2d+3}}{\Gamma\left(\frac{d}{2}\right)^{2d}},
\end{equation}
see \cite[pages 178--179]{Matheron} and \cite[Theorem 16.7.1]{HS}.

The situation changes significantly when moving from the Euclidean space to the hyperbolic space $\mathbb{H}^d$. The negative curvature and the exponential volume growth of that space introduce phase transitions that have no Euclidean counterparts. A systematic study of Poisson hyperplane tessellations in $\mathbb{H}^d$ has attracted renewed attention recently, revealing that the intensity $\gamma$ plays a critical role in the topology of the tessellation cells. A particular feature of hyperbolic Poisson hyperplane tessellations is the existence of a dimension-dependent sharp threshold for the boundedness of the zero cell $Z_o$, which is defined as the almost surely unique cell containing a fixed reference point $o\in\HH^d$, also called the origin. Let
\begin{equation}\label{eq:CriticalIntensity}
	\gamma_c^{(d)} \coloneqq \sqrt{\pi}(d-1) \frac{\Gamma(\frac{d+1}{2})}{\Gamma(\frac{d}{2})}
\end{equation}
denote the critical intensity in dimension $d$. It is known from \cite{BuehlerGusakovaRecke,GodlandKabluchkoThaeleBetaStar,PorretBlanc} that for intensities $\gamma < \gamma_c^{(d)}$, the zero cell $Z_o$ is unbounded with strictly positive probability. Conversely, for $\gamma \geq \gamma_c^{(d)}$, $Z_o$ is almost surely bounded, compare with the simulations shown in Figure \ref{fig:Simulations} for the planar case $d=2$, where $\gamma_c^{(2)} = \pi/2$.

\begin{figure}[t]
    \centering
    \includegraphics[width=0.3\linewidth]{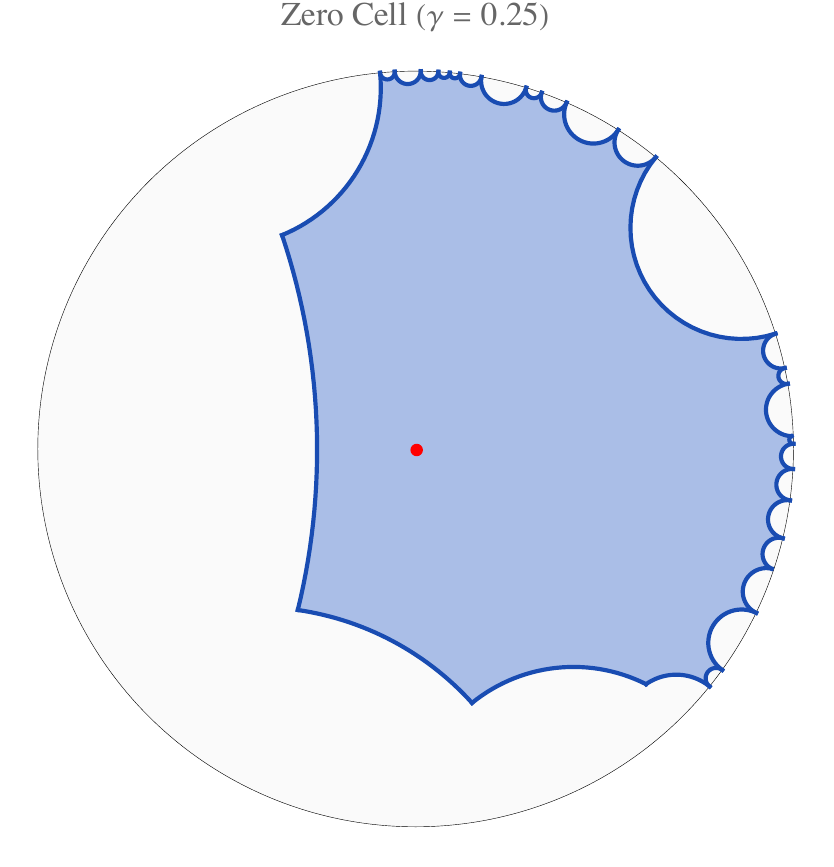}\quad
    \includegraphics[width=0.3\linewidth]{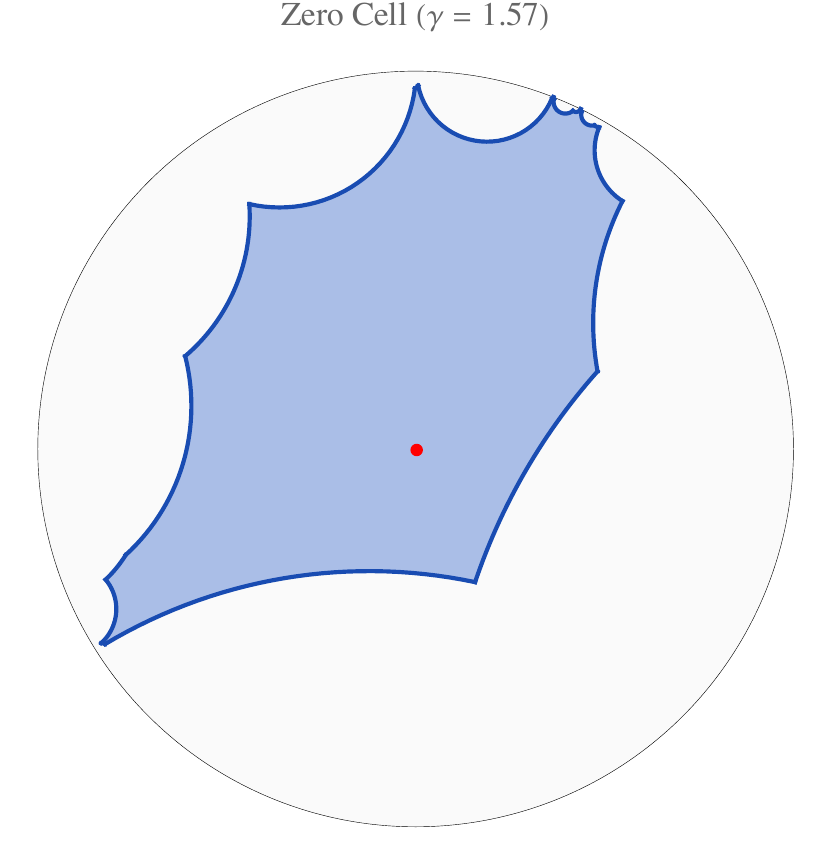}\quad
    \includegraphics[width=0.3\linewidth]{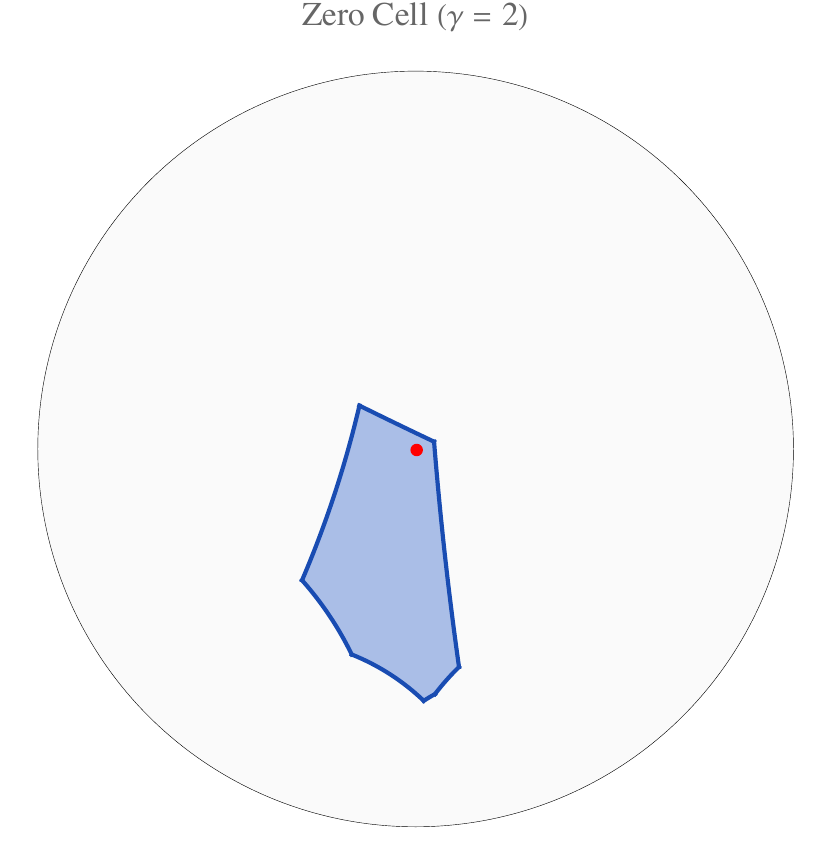}
    \caption{Simulations of Poisson zero cells $Z_o$ for different intensities $\gamma$ in the Poincar\'e disc model for $\HH^2$. Left: $\gamma=0.25$ (subcritical). Middle: $\gamma=\pi/2\approx 1.57$ (critical). Right: $\gamma=2$ (supercritical).}
    \label{fig:Simulations}
\end{figure}

This raises a natural quantitative question: How do volume fluctuations behave at and above the boundedness threshold? In this paper, we focus on the second moment of the volume of the zero cell, $\mathbb{E}\mathcal{H}^d(Z_o)^2$, in hyperbolic space. Here, $\cH^d$ denotes the $d$-dimensional Hausdorff measure on $\HH^d$. The analysis of this second moment is significantly more involved than that of the first moment, dealt with in \cite[Section 7]{BuehlerHugThaeleBM}, due to the need to integrate over pairs of points $(x,y) \in \mathbb{H}^d \times \mathbb{H}^d$. By Fubini's theorem, for fixed $R>0$, the second moment can be expressed as
\begin{equation*}
	\mathbb{E}\mathcal{H}^d(Z_o\cap B_R)^2 = \int_{B_R}\int_{B_R}\mathbb{P}(x,y\in Z_o)\,\mathcal{H}^d(dx)\mathcal{H}^d(dy),
\end{equation*}
where we consider the intersection of $Z_o$ with a hyperbolic ball $B_R$ of radius $R$ centered at $o$ to localize the second moment in the critical regime where moments diverge. 

Our first main result concerns the critical case $\gamma = \gamma_c^{(d)}$. It was established in \cite{PorretBlanc} that at this intensity, although $Z_o$ is bounded almost surely, the expected volume diverges, $\mathbb{E}\mathcal{H}^d(Z_o) = \infty$. Furthermore, if confined to the ball $B_R$, the expected volume $\mathbb{E}\mathcal{H}^d(Z_o \cap B_R)$ grows linearly in $R$, i.e.\ $b_dR\leq\EE\cH^d(Z_o\cap B_R)\leq B_d R$ for sufficiently large $R$ with constants $0<b_d,B_d<\infty$ depending only on $d$, see Corollary 6.2 and Remark 7.2 in \cite{BuehlerHugThaeleBM}. Here, we extend this analysis to the second moment and the variance. We prove that the second moment exhibits polynomial divergence with a universal cubic rate in any dimension.

\begin{theorem}[Critical cubic variance growth]\label{thm:critical}
	Let $d\ge 2$ and $\gamma = \gamma_c^{(d)}$. There are constants $0<c_d,C_d<\infty$ depending only on the dimension such that for all sufficiently large $R$, the second volume moment of the restricted zero cell satisfies
	\[
	c_d R^3 \leq \mathbb{E}\mathcal{H}^d(Z_o\cap B_R)^2 \leq C_d R^3.
	\]
	Consequently, the variance exhibits the same cubic growth.
\end{theorem}

Our second main contribution concerns the supercritical regime $\gamma>\gamma_c^{(d)}$. In this case, the strong decay of the two-point probabilities $\PP(x,y\in Z_o)$ implies that the second volume moment is finite. We derive an exact analytic expression for this moment, which naturally involves the Meijer $G$-function. Recall that the Meijer $G$-function is a classical special function in the theory of hypergeometric functions and integral transforms, see \cite{MathaiSaxena}. It is defined by a Mellin--Barnes type contour integral and will be introduced formally in \eqref{eq:DefGFunction} below. Our derivation relies on tools from harmonic analysis on hyperbolic space. More precisely, we reinterpret the relevant second-moment integral in a spectral form and evaluate it by means of the spherical Fourier transform and the Plancherel formula. This is conceptually parallel to the Euclidean case \cite{HS,Matheron}, where Fourier-analytic methods also play a central role, but the hyperbolic setting is substantially more intricate due to the much more delicate structure of the corresponding harmonic analysis. While harmonic-analytic methods have previously been used in probabilistic models related to hyperbolic spaces, for instance in the study of radial random walks \cite{Voit13}, hyperbolic random connection models \cite{Dickson}, or hyperuniform random measures \cite{BjorglundBylehn}, their systematic use to derive explicit moment formulas in hyperbolic stochastic geometry is new.

\begin{theorem}[Exact supercritical second moment]\label{thm:supercritical}
	Consider a hyperbolic Poisson hyperplane tessellation in $\HH^d$ with intensity $\gamma > \gamma_c^{(d)}$. The second moment of the volume of the zero cell is given by
\[
\EE\cH^d(Z_o)^2\!
=\! \frac{\gamma^3 \pi^{d-2}}{32} \Gamma\left(\frac{d}{2}\right)^2\!\!
G_{10, 10}^{5, 5} \left(
1 \middle| \begin{matrix}
	1-a,\,1-a,\,1-a,\,\frac{5-d}{4},\,\frac{3-d}{4},\,a+\frac{d+1}{2},\,a+\frac{d+1}{2},\,a+\frac{d+1}{2},\,0,\,\frac{1}{2} \\
	a,\,a,\,a,\,\frac{d-1}{4},\,\frac{d+1}{4},\,\frac{1-d}{2}-a,\,\frac{1-d}{2}-a,\,\frac{1-d}{2}-a,\,1,\,\frac{1}{2}
\end{matrix}
\right)\!,
\]
where $a = \frac{1}{4} \big( \frac{\gamma}{\sqrt{\pi}} \frac{\Gamma\left(\frac{d}{2}\right)}{\Gamma\left(\frac{d+1}{2}\right)} - d + 1 \big)$.
\end{theorem}

\begin{figure}[t]
	\centering
	\includegraphics[width=0.49\columnwidth]{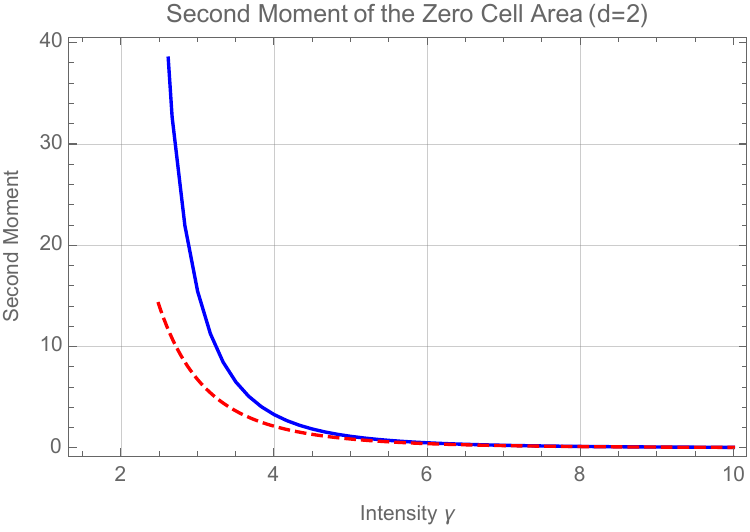}
    \includegraphics[width=0.49\columnwidth]{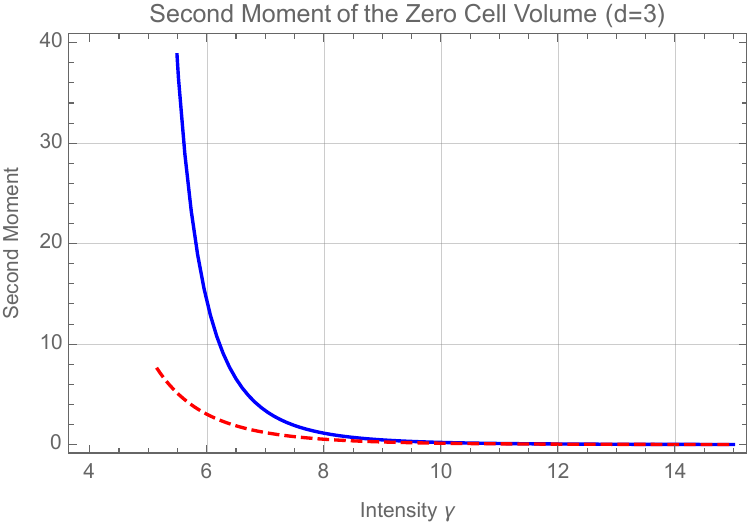}
	\caption{Blue: Exact value of $\EE\cH^2(Z_o)^2$ (left) and $\EE\cH^3(Z_o)^2$ (right) as a function of $\gamma$. Red: The Euclidean value $\frac{4\pi^6}{7\gamma^4}$ (left) and ${14336\pi^2\over\gamma^6}$ (right) as a function of $\gamma$.}
    \label{fig:2ndMom}
\end{figure}

A simplification occurs when the dimension $d$ is odd. For $d = 2k+1$ with $k \in \{1, 2, \dots\}$, the shift parameters in the Meijer $G$-function differ by $1/2$, which allows the use of the duplication formula for the Gamma function, which turns the underlying Mellin–Barnes integrand into a rational function. In what follows, we write $\Res\limits_{z=z_0}f$ for the residue of a meromorphic function \(f\) at the point \(z_0\).

\begin{corollary}[Rational form in odd dimensions]\label{cor:odd_dimensions}
	Let \(d=2k+1\) for some integer \(k\ge 1\), and let $a$ be as in Theorem \ref{thm:supercritical}.
	Then
	\[
	\EE \cH^{2k+1}(Z_o)^2
	=
	2\pi^{2k-1}\Gamma\!\left(k+\frac12\right)^2\left(\frac{\gamma}{4}\right)^3
	\sum_{m=0}^{k}
	\Res\limits_{s=-(a+m)} \frac{P_k(s)}{Q_k(s)},
	\]
	where
	\[
	P_k(s)\coloneqq\prod_{j=0}^{k-1}\left(\frac{j^2}{4}-s^2\right),
	\qquad
	Q_k(s)\coloneqq\prod_{\ell=0}^{k}\bigl((a+\ell)^2-s^2\bigr)^3.
	\]
\end{corollary}

For example, if $d=3$ ($k=1$) and $d=5$ ($k=2$) we obtain\footnote{The symbolic computations for the exact evaluation of the residues in Corollary \ref{cor:odd_dimensions} were carried out using Wolfram Mathematica.}
\begin{align*}
	\EE \cH^{3}(Z_o)^2
	&=
	\frac{2048\pi^2\bigl(7\gamma^2-48\bigr)}
	{\gamma^2\bigl(\gamma^2-16\bigr)^3},\\
\EE \cH^{5}(Z_o)^2
&=
\frac{2^{40}\pi^4\bigl(891\gamma^4-50688\gamma^2+524288\bigr)}
{3^{13}\,\gamma^2\bigl(\gamma^2-\frac{256}{9}\bigr)^3\bigl(\gamma^2-\frac{1024}{9}\bigr)^3}.
\end{align*}

While expressed in terms of a special function, the exact formula of Theorem \ref{thm:supercritical} allow us to extract the asymptotics as $\gamma\to\infty$ and as $\gamma\downarrow\gamma_c^{(d)}$. In what follows, we write $f_1(\gamma)\sim f_2(\gamma)$ for two functions $f_1$ and $f_2$, provided that $f_1(\gamma)/f_2(\gamma)\to 1$ as $\gamma$ approaches the relevant limit.

\begin{corollary}[Euclidean limit]\label{cor:GammaToInfinity_d}
	Asymptotically, as the intensity $\gamma \to \infty$, the second moment of the volume of the zero cell decays according to the power law
	\[
	\EE\cH^d(Z_o)^2 \sim C_d \gamma^{-2d},\qquad C_d = 2^{4d} \pi^{\frac{4d-3}{2}} \frac{\Gamma\left(d+\frac{3}{2}\right)}{\Gamma\left(\frac{3d+3}{2}\right)} \frac{\Gamma\left(\frac{d+1}{2}\right)^{2d+3}}{\Gamma\left(\frac{d}{2}\right)^{2d}}.
	\]
\end{corollary}

We emphasize that the limiting value on the right-hand side coincide exactly with the known second volume moment \eqref{eq:SecondMomentEuclidean} of the zero cell in a Euclidean Poisson hyperplane tessellation in $\mathbb{R}^d$. This convergence is geometrically expected. As the intensity $\gamma \to \infty$, the zero cell $Z_o$ contracts to the origin almost surely. Since the hyperbolic space $\mathbb{H}^d$ is a Riemannian manifold and thus locally Euclidean, the influence of the negative curvature becomes negligible on the vanishing scale of the cell, recovering the Euclidean values in the limit.

We turn now to the regime in which $\gamma$ approaches the critical intensity from above. It was shown in \cite[Remark 7.2]{BuehlerHugThaeleBM} that 
\begin{equation}\label{eq:AsymptFirstMom}
\EE\cH^d(Z_o) \sim \frac{\pi^{{d+1\over 2}}\,\Gamma(\frac{d+1}{2})}{2^{d-2}\,\Gamma(\frac{d}{2})^2} \left( \gamma - \gamma_c^{(d)} \right)^{-1}. 
\end{equation}
as $\gamma\downarrow\gamma_c^{(d)}$. Here, we extend this result to the second moment.

\begin{corollary}[Critical divergence]\label{cor:GammaToCritical_d}
	As the intensity approaches the critical value from the supercritical side, i.e.\ $\gamma \downarrow \gamma_c^{(d)}$, the second moment of the volume of the zero cell diverges according to
	\[
	\EE\cH^d(Z_o)^2 \sim K_d \left( \gamma - \gamma_c^{(d)} \right)^{-3},\qquad K_d = \frac{\pi^{d+1}}{2^{d-2}} (d-1) \frac{\Gamma\left(\frac{d+1}{2}\right)^2}{\Gamma\left(\frac{d}{2}\right)^4}.
	\]
\end{corollary}

The asymptotics \eqref{eq:AsymptFirstMom} together with Corollary \ref{cor:GammaToCritical_d} show that the first two volume-moment exponents of the zero cell coincide with the classical mean-field exponents from percolation theory. Viewing the zero cell $Z_o$ as the analogue of the cluster at the origin, the quantities $\EE\cH^d(Z_o)$ and $\EE\cH^d(Z_o)^2$ correspond to the first and second moment of the cluster size. In mean-field models, such as percolation on a Cayley tree (or Bethe lattice), this first and second moment are known to diverge with integer exponents $1$ and $3$, respectively. In fact, this is a direct consequence of the cluster size distribution derived around \cite[Equation (10.9)]{Grimmett}. In contrast, low-dimensional Euclidean percolation models exhibit fractional exponents, such as the susceptibility exponent $43/18$ in the planar case \cite[Table 10.1]{Grimmett}. The appearance of the mean-field exponents $1$ and $3$ in \eqref{eq:AsymptFirstMom} and Corollary \ref{cor:GammaToCritical_d} reflect the exponential growth of the hyperbolic space, which asymptotically suppresses spatial correlations in a manner analogous to tree-like structures.

\smallskip

Our paper is organized as follows. Section \ref{sec:preliminaries} provides the necessary background on hyperbolic geometry and Poisson hyperplane tessellations. In that section, we also restate the facts from harmonic analysis in hyperbolic space which are essential to our computation. In Section \ref{sec:CriticalProof}, we prove the cubic divergence of the second moment in the critical case (Theorem \ref{thm:critical}). Finally, Section \ref{sec:2ndMomentExact} derives the exact formula for the supercritical second moment (Theorem \ref{thm:supercritical} and Corollary \ref{cor:odd_dimensions}) together with its asymptotic behaviour (Corollaries \ref{cor:GammaToInfinity_d} and \ref{cor:GammaToCritical_d}).

\section{Preliminaries}\label{sec:preliminaries}

\subsection{Hyperbolic space}

In this paper we work with the hyperbolic space $\mathbb{H}^d$ of dimension $d \ge 2$, which is the unique simply connected, complete, $d$-dimensional Riemannian manifold with constant sectional curvature $-1$. We denote by $d_h(\,\cdot\,,\,\cdot\,)$ the geodesic distance metric on $\mathbb{H}^d$. For a fixed origin $o \in \mathbb{H}^d$ and a radius $R>0$, let $B_R \coloneqq \{x \in \mathbb{H}^d : d_h(o,x) \le R\}$ be the closed hyperbolic ball centered at $o$.

The Riemannian volume measure on $\mathbb{H}^d$ is denoted by $\cH^d$. It is the $d$-dimensional Hausdorff measure on $\HH^d$ induced by the metric $d_h$. In our calculations, we frequently rely on hyperbolic polar coordinates. Any point $x \in \mathbb{H}^d \setminus \{o\}$ can be uniquely represented as $x=(r, u)$, where $r = d_h(o,x) > 0$ is the radial distance from the origin and $u \in \mathbb{S}^{d-1}$ is a direction vector on the unit sphere (in the tangent space at $o$). In these coordinates, the hyperbolic volume measure takes the form
\begin{equation}\label{eq:PolarCoordinates}
\cH^d(dx) = \sinh^{d-1}(r) \, dr \, d\sigma_{d-1}(u),
\end{equation}
where $\sigma_{d-1}$ denotes the spherical Lebesgue measure on $\mathbb{S}^{d-1}$, normalized such that
$$
\omega_d\coloneqq\sigma_{d-1}(\SS^{d-1})={2\pi^{d/2}\over\Gamma(d/2)}.
$$

Let $A_h(d, d-1)$ denote the space of hyperplanes in $\mathbb{H}^d$. Each element $H \in A_h(d, d-1)$ is a complete totally geodesic hypersurface in $\mathbb{H}^d$, which divides the space into two open half-spaces. We equip the space $A_h(d, d-1)$ with the measure $\mu$ that is invariant under the action of the group of isometries of $\mathbb{H}^d$. This measure is unique up to a normalization factor. We choose to normalize $\mu$ such that for any geodesic segment $S \subset \mathbb{H}^d$ of length $r>0$, the measure of the set of hyperplanes intersecting $S$ is given by
\begin{equation*}
	\mu(\{H \in A_h(d, d-1) : H \cap S \neq \emptyset\}) = \frac{1}{\sqrt{\pi}} \frac{\Gamma(\frac{d}{2})}{\Gamma(\frac{d+1}{2})}\, r.
\end{equation*}
We remark that our choice for the normalization of $\mu$ is consistent with the one used earlier in \cite{BuehlerHugThaeleBM,GodlandKabluchkoThaeleBetaStar,HeroldHugThaele,HS}. Using the additivity of the measure $\mu$ and the fact that, for almost every hyperplane the intersection with the boundary of a hyperbolic triangle \(T\) consists of precisely two points, this normalization implies that for a hyperbolic triangle \(T\) with side lengths \(a,b,c\), the measure of the set of hyperplanes intersecting \(T\) is
\begin{equation}\label{eq:CroftonTriangle}
\mu(\{H \in A_h(d, d-1) : H \cap T \neq \emptyset\}) = \frac{1}{2\sqrt{\pi}} \frac{\Gamma(\frac{d}{2})}{\Gamma(\frac{d+1}{2})} (a+b+c).    
\end{equation}
Further background material on hyperbolic geometry can be found in the monograph \cite{Ratcliffe}.

\subsection{Harmonic analysis on hyperbolic space}

To evaluate the integral arising in the computation of the second volume moment and to show that it is finite, we interpret it within the framework of harmonic analysis on the hyperbolic space $\HH^d$, for which we refer to \cite{Faraut,Helgason,Koornwinder,Strichartz,Wolf}. Because the results we need are distributed over several sources and formulated with differing conventions and normalizations, we briefly review the necessary material. Among these references, the lecture notes \cite{Faraut} of Faraut is perhaps the closest to our present needs, but they are written in French and not widely available. The classical monograph of Helgason \cite{Helgason} gives a detailed treatment in dimension two, but works with curvature $-4$ instead of the standard curvature $-1$. The higher-dimensional case can then be obtained from the general group-theoretic results developed there. Koornwinder's survey \cite{Koornwinder} is also very helpful, but uses a different normalization of the radial measure, and similar discrepancies in normalization appear also in the paper of Strichartz \cite{Strichartz} and the monograph of Wolf \cite{Wolf}. Strichartz gives a representation of the spherical functions in terms of associated Legendre functions, which is particularly useful to us as it allows us to explicitly calculate the spherical transform of $x \mapsto e^{-c \, d_h(o,x)}$, $c>0$, using integral identities available in the literature. For the convenience of the reader, we therefore compile the relevant facts here and indicate how they must be adjusted to fit our framework.

We identify $\HH^d$ with the symmetric space $G/K$, where $G = SO_0(d,1)$ acts transitively on $\HH^d$ by orientation preserving isometries and $K = SO(d)$ is the stabilizer of the origin $o \in \HH^d$. Here $O(d,1)$ is the generalized Lorentz group, which consists of linear transformations that preserve the quadratic form $(x_1,\ldots,x_{d+1})\mapsto x_1^2+\ldots+x_d^2-x_{d+1}^2$ with signature $(d,1)$ and $SO_0(d,1)$ is the connected component of the identity of that group. In this identification, points $x, y \in \HH^d$ are represented by cosets $gK, hK$ in $G$, and the hyperbolic distance $d_h(x,y)$ is invariant under the diagonal action of $G$, satisfying $d_h(gx, gy) = d_h(x,y)$ for all $g \in G$.

We call a function $f \colon \HH^d \to \RR$ \emph{radial} when $f(x)$ depends only on the distance of $x$ to the origin $o$.
Note that this is equivalent to $f$ being $K$-invariant, i.e., $f(k\cdot x) = f(x)$ for all $x \in \HH^d$, $k\in K$.
For a radial function $f$, we will sometimes write, in a slight abuse of notation,  $f(r) \coloneqq f(x)$ when $x\in\HH^d$ satisfies $d_h(o,x)=r\geq 0$.
For $x,y \in \HH^d$, we introduce the notation
\begin{equation}\label{eq:def_inverse}
    f(x^{-1}y) \coloneqq f(g_x^{-1}g_y \cdot o),
\end{equation}
where $g_x,g_y \in G$ satisfy $g_x \cdot o = x$, $g_y \cdot o = y$.
Note that the definition does not depend on the choice of $g_x,g_y$, as any other choice $\tilde g_x = g_x k_1$, $\tilde g_y = g_y k_2$ leads to the same value
\[ f(\tilde g_x^{-1} \tilde g_y^{-1} \cdot o) 
= f(k_1^{-1} g_x^{-1} g_y k_2 \cdot o) 
= f(k_1^{-1} g_x^{-1} g_y \cdot o)
= f(g_x^{-1} g_y \cdot o), \]
where the $K$-invariance of $f$ enters in the last equation.
The convolution of two radial functions $f,g$ is defined by
\[ (f*g )(x) \coloneqq \int_{\HH^d} f(y^{-1}x) g(y)\, \cH^{d}(dy), \qquad x \in \HH^d,\]
whenever this integral exists for almost all $x$.
In particular, this includes the case where $f \in L^1(\HH^d)$ and $g \in L^p(\HH^d)$ for some $p \in [1,\infty]$, as well as the case $f,g \in L^2(\HH^d)$. Strictly speaking, one does not need to assume that both functions are radial. However, since only the radial case is relevant for our purposes, we restrict to this setting throughout.

A deep result of non-Euclidean harmonic analysis is the \emph{Kunze--Stein phenomenon} (see \cite{Cowling78}), which has the following consequence in hyperbolic space.

\begin{proposition}\label{prop:kunze_stein}
    Let $1\leq p < 2$ and let $f \in L^p(\HH^d)$ and $g \in L^2(\HH^d)$ be radial functions.
    Then $f * g \in L^2(\HH^d)$.
\end{proposition}

Note that this is in stark contrast with $\RR^d$, where for $p \in (1,2)$ the convolution of an $L^p$-function and an $L^2$-function need not be an $L^2$-function in general.
We remark that the main result of \cite{Cowling78} is in fact much more general, the above proposition follows as a corollary when specialising to $SO(d)$-invariant functions on $SO_0(d,1)$.

\medspace

The \emph{spherical transform} of a radial function $f\in L^1(\HH^d)$ is defined by the integral
\[
\widehat f(\lambda) \coloneqq \int_{\HH^d} f(x)\,\phi_{-\lambda}(x)\,\cH^d(dx),\qquad \lambda\in\RR.
\]
Here, the elementary spherical function $\phi_\lambda$ is the radial eigenfunction of the Laplace--Beltrami operator on $\HH^d$ corresponding to the eigenvalue $-(d-1)^2/4-\lambda^2$, normalized in such a way that $\phi_\lambda(o)=1$, see \cite[Equation (1.2')]{Strichartz} or \cite[page 357]{Faraut}. In the notation of \cite{Koornwinder}, this is precisely the Jacobi function $\phi_\lambda^{(\alpha,\beta)}$ with $(\alpha,\beta)=({d\over 2}-1,-{1\over 2})$,
since for this choice the differential operator $L_{\alpha,\beta}$ in \cite{Koornwinder} coincides with the radial part of the Laplace--Beltrami operator on $\HH^d$. We use the normalization $\phi_\lambda(o)=1$, so compared to \cite[Equations (4.3) and (4.5)]{Strichartz} we divide out the factor $\omega_d\gamma_d(\lambda)$ appearing in the spectral projection kernel of \cite{Strichartz}. With this normalization, \cite[Equation (4.5)]{Strichartz} yields the representation
\[
\phi_\lambda(x)
=K(d)\,\sinh^{\frac{2-d}{2}}(d_h(o,x))
P_{-\frac12+i\lambda}^{\frac{2-d}{2}}\!\bigl(\cosh(d_h(o,x))\bigr),
\]
where $K(d)=2^{\frac{d-2}{2}}\Gamma\!\left(\frac d2\right)$, $i=\sqrt{-1}$ is the imaginary unit and $P_b^a$ are the Legendre functions of the first kind. The value of $K(d)$ is chosen so that $\phi_\lambda(o)=1$, using
\[
\lim_{r\to 0+}\sinh^a(r)\,P_b^a(\cosh(r)) = \frac{2^a}{\Gamma(1-a)}
\]
for $a\in\RR\setminus\NN$ and $b\in\CC$, which in turn follows from \cite[Formula 14.8.7]{NIST} and the fact that $\sinh^2(r)=(\cosh(r)+1)(\cosh(r)-1)$. Since $\phi_{-\lambda}=\phi_\lambda$ for $\lambda\in\RR$, we can switch between both functions in the explicit calculations below. Representations of $\phi_\lambda$ in terms of the Gauss hypergeometric function can be found, for example, in \cite{Koornwinder,Strichartz,Wolf}.

By Plancherel's theorem, the spherical transform admits a unique isometric extension from the radial subspace of $L^1(\HH^d)\cap L^2(\HH^d)$ to the radial subspace of $L^2(\HH^d)$, mapping it onto $L^2(\RR_+,\nu)$, where
\[
d\nu(\lambda)=|c(\lambda)|^{-2}\,d\lambda
\]
is the Plancherel measure associated with the Harish--Chandra $c$-function, see \cite[Theoreme IV.2]{Faraut}.
In particular, for radial $f,g\in L^2(\HH^d)$ one has
\begin{equation}\label{eq:Plancherel}
	\langle f,g\rangle_{L^2(\HH^d)}
	=
	\int_0^\infty \widehat f(\lambda)\,\overline{\widehat g(\lambda)}\,|c(\lambda)|^{-2}\,d\lambda.
\end{equation}

We now explain how to extract an explicit representation of \(c(\lambda)\) from \cite{Koornwinder}.
Let \((\alpha,\beta)=\left(\frac{d}{2}-1,-\frac{1}{2}\right)\) be as above, and let \(\phi^{(\alpha,\beta)}_\lambda\), \(\lambda\geq 0\), denote the corresponding Jacobi function defined in \cite[Equation (2.4)]{Koornwinder}.
Recall that, for \(\lambda\geq 0\), the function \(\phi^{(\alpha,\beta)}_\lambda\) agrees with the spherical function \(\phi_\lambda\) introduced above, in the sense that
\[
\phi_\lambda(x)=\phi^{(\alpha,\beta)}_\lambda(d_h(o,x)),
\qquad x\in\HH^d.
\]
Moreover, for this choice of \(\alpha\) and \(\beta\), the function \(\Delta=\Delta_{\alpha,\beta}\) from \cite[Equation (2.5)]{Koornwinder} is given by
\[
\Delta(t)=2^{d-1}\sinh^{d-1}(t),
\qquad t>0.
\]
Finally, for our fixed choice of \(\alpha\) and \(\beta\), we write \(\widetilde F\) for the Jacobi transform of an even function \(F\in C_c^\infty(\RR)\) as introduced in \cite[Equation (2.12)]{Koornwinder}. Here \(C_c^\infty(\RR)\) denotes the space of infinitely differentiable functions on \(\RR\) with compact support.

Now let \(F,G\) be even functions in \(C_c^\infty(\RR)\), and let \(f,g\) be the associated radial functions on \(\HH^d\) defined by
\[
f(x)\coloneqq F(d_h(o,x)),
\qquad
g(x)\coloneqq G(d_h(o,x)).
\]
Passing to polar coordinates using \eqref{eq:PolarCoordinates} and recalling that \(\phi_{-\lambda}=\phi_\lambda\), we obtain for \(\lambda\geq 0\),
\[
\widehat f(\lambda)
= \omega_d \int_0^\infty F(r)\phi_\lambda^{(\alpha,\beta)}(r)\sinh^{d-1}(r)\,dr
= \frac{\omega_d}{2^{d-1}} \int_0^\infty F(r)\phi_\lambda^{(\alpha,\beta)}(r)\Delta(r)\,dr
= \frac{\omega_d}{2^{d-1}}\,\widetilde F(\lambda),
\]
and similarly $\widehat g(\lambda)=\frac{\omega_d}{2^{d-1}}\,\widetilde G(\lambda)$. Moreover, another application of polar coordinates yields
\[
\langle f,g\rangle_{L^2(\HH^d)}
= \omega_d \int_0^\infty F(r)G(r)\sinh^{d-1}(r)\,dr
= \frac{\omega_d}{2^{d-1}} \int_0^\infty F(r)G(r)\Delta(r)\,dr.
\]
By \cite[Theorem 2.4]{Koornwinder} (note that the set \(D_{\alpha,\beta}\) appearing there in the definition of \(\nu\) is empty), this is equal to
\[
\frac{\omega_d}{2^{d-1}}
\int_0^\infty
\widetilde F(\lambda)\widetilde G(\lambda)\,
\frac{|c_{\alpha,\beta}(\lambda)|^{-2}}{2\pi}\,d\lambda
=
\frac{2^{d-1}}{\omega_d}
\int_0^\infty
\widehat f(\lambda)\widehat g(\lambda)\,
\frac{|c_{\alpha,\beta}(\lambda)|^{-2}}{2\pi}\,d\lambda,
\]
where, by \cite[Equation (2.18)]{Koornwinder},
\[
c_{\alpha,\beta}(\lambda)
=
\frac{2^{(d-1)/2-i\lambda}\Gamma(\frac{d}{2})\Gamma(i\lambda)}
{\Gamma\!\left(\frac{i\lambda+(d-1)/2}{2}\right)
\Gamma\!\left(\frac{i\lambda+(d-1)/2+1}{2}\right)}.
\]

We now apply Legendre's duplication formula
\begin{equation}\label{eq:Legendre}
\Gamma(z)\Gamma\!\left(z+\frac12\right)
=
2^{1-2z}\sqrt{\pi}\,\Gamma(2z),
\qquad z\in\CC,
\end{equation}
with $z=\frac{i\lambda+(d-1)/2}{2}$.
This gives
\[
\Gamma\!\left(\frac{i\lambda+(d-1)/2}{2}\right)
\Gamma\!\left(\frac{i\lambda+(d-1)/2+1}{2}\right)
=
2^{1-i\lambda-(d-1)/2}\sqrt{\pi}\,
\Gamma\!\left(i\lambda+\frac{d-1}{2}\right),
\]
and therefore
\[
c_{\alpha,\beta}(\lambda)
=
\frac{2^{(d-1)/2-i\lambda}\Gamma(\frac{d}{2})\Gamma(i\lambda)}
{2^{1-i\lambda-(d-1)/2}\sqrt{\pi}\,\Gamma(i\lambda+(d-1)/2)}
=
\frac{2^{d-2}\Gamma(\frac{d}{2})}{\sqrt{\pi}}\,
\frac{\Gamma(i\lambda)}{\Gamma\!\left(\frac{d-1}{2}+i\lambda\right)},
\]
see also the formula on the bottom of page 357 in \cite{Faraut} and take $s=i\lambda$ there. Collecting the constants, we conclude that \eqref{eq:Plancherel} holds with
\begin{align}\label{eq:Cfunction}
c(\lambda)
=
\frac{2^{1-d}\pi^{-d/2}}{\Gamma(d/2)}
\left|
\frac{\Gamma\!\left(\frac{d-1}{2}+i\lambda\right)}{\Gamma(i\lambda)}
\right|^2.
\end{align}

Moreover, the convolution identity
\begin{equation}\label{eq:convolution_identity}
    \widehat{(f*g)} = \widehat{f} \cdot \widehat{g}.   
\end{equation}
holds for radial functions $f,g \in L^1(\HH^d)$, cf.\ \cite[page 411]{Faraut} or \cite[page 43]{Helgason}.
Together with the Kunze--Stein phenomenon described in Proposition \ref{prop:kunze_stein}, we get the following result.

\begin{lemma}\label{lem:convolution}
Let $1\leq p<2$. Let $f\in L^p(\HH^d)\cap L^2(\HH^d)$ and $g\in L^2(\HH^d)$ be non-negative radial functions.
Then
\[
\widehat{f*g}=\widehat f \cdot \widehat g
\qquad\text{in }L^2(\RR_+,\nu),
\]
that is, the identity holds for $\nu$-almost every $\lambda\in\RR_+$.
\end{lemma}
\begin{proof}
For $R>0$, define
\[
f_R\coloneqq f\,\mathbf 1_{B(o,R)},
\qquad
g_R\coloneqq g\,\mathbf 1_{B(o,R)}.
\]
Since $B(o,R)$ has finite hyperbolic volume and $f\in L^2(\HH^d)$, $g\in L^2(\HH^d)$, we have
\[
f_R,g_R\in L^1(\HH^d)\cap L^2(\HH^d).
\]
Moreover, $f_R$ and $g_R$ are radial. Hence the  convolution identity \eqref{eq:convolution_identity} for the spherical transform yields
\[
\widehat{f_R*g_R}=\widehat{f_R}\,\widehat{g_R}.
\]

Since $f_R\uparrow f$ and $g_R\uparrow g$ pointwise and all functions are non-negative, we have for a.e.\ $x\in\HH^d$,
\[
0\le (f_R*g_R)(x)\uparrow (f*g)(x)
\]
by the monotone convergence theorem. Because $f*g\in L^2(\HH^d)$ by Proposition \ref{prop:kunze_stein}, it follows that
\[
|f_R*g_R-f*g|^2\le |f*g|^2\in L^1(\HH^d).
\]
Therefore, by dominated convergence,
\[
f_R*g_R\to f*g \qquad\text{in }L^2(\HH^d).
\]

Also, since $f_R\to f$ and $g_R\to g$ in $L^2(\HH^d)$, the Plancherel theorem for the spherical transform implies
\[
\widehat{f_R*g_R}\to \widehat{f*g},\qquad
\widehat{f_R}\to \widehat f,\qquad
\widehat{g_R}\to \widehat g
\qquad\text{in }L^2(\RR_+,\nu).
\]
Passing to a common subsequence $R_n\to\infty$, we may assume that all three convergences hold pointwise for $\nu$-a.e.\ $\lambda\in\RR_+$. For such $\lambda$, we obtain
\[
\widehat{f*g}(\lambda)
=\lim_{n\to\infty}\widehat{f_{R_n}*g_{R_n}}(\lambda)
=\lim_{n\to\infty}\widehat{f_{R_n}}(\lambda)\,\widehat{g_{R_n}}(\lambda)
=\widehat f(\lambda)\,\widehat g(\lambda).
\]
This proves the claim.
\end{proof}

\begin{remark}
    The above identity also holds without the condition $f,g \geq 0$, however, the proof becomes more technical. Since our applications focus exclusively on non-negative radial functions, this case suffices for our purposes.
\end{remark}

\subsection{Hyperbolic Poisson hyperplane processes}

Let $\eta$ be a Poisson process on $A_h(d,d-1)$ with intensity measure $\Lambda=\gamma\mu$, where $\gamma>0$ is the intensity parameter. This means that for every measurable set $B\subseteq A_h(d,d-1)$, the random variable $\eta(B)$ is Poisson distributed with mean $\Lambda(B)=\gamma\mu(B)$, and that for pairwise disjoint measurable sets $B_1,\dots,B_n\subseteq A_h(d,d-1)$, the random variables $\eta(B_1),\dots,\eta(B_n)$ are independent.
The random collection of hyperplanes induced by $\eta$ partitions the hyperbolic space into a countable family of random convex polyhedra.
In this paper, we follow the terminology of \cite{Ratcliffe}, in particular Sections~6.3 and~6.5. Thus, a convex polyhedron is a nonempty, closed, convex subset of $\HH^d$ such that only finitely many of its facets intersect any bounded subset of $\HH^d$.
A polytope is a compact convex polyhedron, see \cite[Theorem~6.5.1]{Ratcliffe}.
Note that the cells induced by $\eta$ are not necessarily compact in general.
We call the random subdivision of $\HH^d$ induced by $\eta$ a hyperbolic Poisson hyperplane tessellation.

Our main object of study is the zero cell $Z_o$ of this tessellation, defined as the almost surely unique cell containing the previously fixed origin $o$. Formally, $Z_o$ is the intersection of all half-spaces defined by the hyperplanes of $\eta$ that contain $o$:
\[
Z_o \coloneqq \bigcap_{H \in \eta} H^{-},
\]
where $H^{-}$ denotes the closed half-space bounded by $H$ that contains $o$.

\section{Volume of the critical truncated zero cell: Proof of Theorem \ref{thm:critical}}\label{sec:CriticalProof}

Consider a Poisson hyperplane tessellation in $\HH^d$ with intensity $\gamma$. We denote the zero cell of this tessellation by $Z_o$ and investigate the second volume moment $\EE\cH^d(Z_o\cap B_R)^2$ of the truncated zero cell $Z_o \cap B_R$, where $B_R$ is a hyperbolic ball centred at $o$ with radius $R > 0$.
By Fubini's theorem, we have
\[
\EE\cH^d(Z_o\cap B_R)^2 = \int_{B_R}\int_{B_R}\PP(x,y\in Z_o)\,\cH^d(dx)\cH^d(dy).
\]
The event $\{x,y\in Z_o\}$ is equivalent to the event that the hyperbolic triangle $\Delta(o,x,y)$ with vertices $o$, $x$, $y$ is not hit by any of the hyperplanes of the Poisson hyperplane process.
This yields
\[ \PP(x,y\in Z_o) = e^{-\gamma\mu_{d-1}([\Delta(o,x,y)])}, \]
where we denote by $[\Delta(o,x,y)]=\{H\in A_h(d,d-1):H\cap \Delta(o,x,y)\neq\varnothing\}$ the set of hyperplanes having nonempty intersection with $\Delta(o,x,y)$.
By \eqref{eq:CroftonTriangle}, we have that
\[ \mu_{d-1}([\Delta(o,x,y)]) 
= c_d \big(d_h(o,x)+d_h(o,y)+d_h(x,y)\big) \]
with the constant
\begin{equation}\label{eq:def_cd}
c_d 
\coloneqq \frac{1}{2\sqrt{\pi}}  \frac{\Gamma(\frac{d}{2})}{\Gamma(\frac{d+1}{2})} 
= \frac{d-1}{2\gamma_c^{(d)}}.     
\end{equation}

We thus obtain
\begin{equation}\label{eq:TruncatedMomentIntegralRep}
\EE\cH^d(Z_o\cap B_R)^2 = \int_{B_R}\int_{B_R} e^{-\gamma\, c_d\,(d_h(o,x)+d_h(o,y)+d_h(x,y))}\,\cH^d(dx)\cH^d(dy).
\end{equation}
Let $\theta$ denote the angle of the hyperbolic triangle $\Delta(o,x,y)$ at $o$ and let $s\coloneqq d_h(o,x)$ and $t\coloneqq d_h(o,y)$ denote the two adjacent side lengths.
By the hyperbolic law of cosines \cite[Theorem 2.5.3]{Ratcliffe},
\[ \cosh d_h(x,y) = \cosh(s)\cosh(t)-\sinh(s)\sinh(t)\cos(\theta). \]
By applying the triangle- and reverse triangle inequality to $\Delta(o,x,y)$, we get the upper and lower bounds
\begin{equation}\label{eq:upper_lower_dist}
    |s-t| \leq  d_h(x,y) \leq s+t.
\end{equation}
Using hyperbolic polar coordinates as in \eqref{eq:PolarCoordinates}, the above identity and \eqref{eq:TruncatedMomentIntegralRep}, we obtain
\begin{align*}
\EE\cH^d(Z_o\cap B_R)^2
&= \int_0^R \int_0^R \int_{\SS^{d-1}} \int_{\SS^{d-1}} 
e^{-\gamma\, c_d \, (s+t+d(s,t,\theta(u,v)))}\\
&\hspace{4cm}\cH^{d-1}(du) \cH^{d-1}(dv)\, \sinh^{d-1}(s)\,ds\, \sinh^{d-1}(t)\,dt,    
\end{align*}
where we write $\theta(u,v)$ for the angle between $u$ and $v$ and use the notation
\[ d(s,t,\theta) \coloneqq \arcosh\big(\cosh(s)\cosh(t)-\sinh(s)\sinh(t)\cos(\theta)\big). \]
To simplify the above expression, note that for any $s,t > 0$ the innermost integral is independent of the choice of $v \in \SS^{d-1}$ and equals
\begin{align*}
\int_{\SS^{d-1}} e^{-\gamma\, c_d \, (s+t+d(s,t,\theta(u,v)))} \,\cH^{d-1}(du)
= \int_{\SS^{d-1}}e^{-\gamma\, c_d \, (s+t+d(s,t,\theta(u,e_1)))} \,\cH^{d-1}(du),
\end{align*}
where $e_1\coloneqq(1,0,\ldots,0)\in\RR^d$ denotes the first unit vector in $\RR^d$.

Let $d \geq 3$ and $f \colon \SS^{d-1} \to [0,\infty)$ be a measurable function. The Funk-Hecke formula, in its simplest form, says that
$$
\int_{\SS^{d-1}} f(\langle u,e_1\rangle)\,\cH^{d-1}(du) = {2\pi^{d-1\over 2}\over\Gamma({d-1\over 2})}\int_{-1}^{+1}f(h)(1-h^2)^{d-3\over 2}\,dh,
$$
where $\langle\,\cdot\,,\,\cdot\,\rangle$ is the Euclidean standard scalar product, see \cite[Equation (2.6)]{KSTBook} or \cite[Theorem 6]{MuellerBook}. In particular, substituting $h=\cos\theta$, we obtain
\begin{equation*}
    \int_{\SS^{d-1}} f(\langle u,e_1\rangle)\, \cH^{d-1}(du) 
    = {2\pi^{d-1\over 2}\over\Gamma({d-1\over 2})} \int_0^\pi f(\cos(\theta)) \sin^{d-2}(\theta) \,d\theta.
\end{equation*}
Using this identity, we get
\begin{align*}
\int_{\SS^{d-1}}e^{-\gamma\, c_d \, (s+t+d(s,t,\theta(u,e_1)))} \,\cH^{d-1}(du) &= \int_{\SS^{d-1}}e^{-\gamma\, c_d \, (s+t+d(s,t,\langle u,e_1\rangle))} \,\cH^{d-1}(du)\\
&= {2\pi^{d-1\over 2}\over\Gamma({d-1\over 2})} \int_0^\pi e^{-\gamma\, c_d \, (s+t+d(s,t,\theta))} \sin^{d-2}(\theta)\, d\theta.
\end{align*}
In summary, we arrive at
\[ \EE\cH^d(Z_o\cap B_R)^2
= {4\pi^{2d-1\over 2}\over\Gamma({d\over 2})\Gamma({d-1\over 2})} \int_0^R \int_0^R \int_0^\pi 
e^{-\gamma\, c_d \, (s+t+d(s,t,\theta))} \sin^{d-2}(\theta)\, d\theta \sinh^{d-1}(s)\,ds \,\sinh^{d-1}(t)\,dt. \]

From now on we will focus on the critical case $\gamma = \gamma_c^{(d)}$, where $\gamma_c^{(d)}$ is given by \eqref{eq:CriticalIntensity}. Noting that $c_d \cdot \gamma_c^{(d)} = \frac{d-1}{2}$ and taking $\gamma = \gamma_c^{(d)}$ we arrive at
\begin{equation}\label{eq:TruncSecMomInt}
\EE\cH^d(Z_o\cap B_R)^2
= {4\pi^{2d-1\over 2}\over\Gamma({d\over 2})\Gamma({d-1\over 2})} \int_0^R \int_0^R \int_0^\pi 
g(s,t,\theta)\, d\theta ds dt
\end{equation}
with
\[ g(s,t,\theta) \coloneqq e^{- \frac{d-1}{2}(s+t+d(s,t,\theta))} \sin^{d-2}(\theta) \sinh^{d-1}(s)\sinh^{d-1}(t),\qquad s,t \geq 0, \theta \in [0,\pi]. \]
Ignoring the constant prefactor, we now show that
\[ I(R) \coloneqq \int_0^R \int_0^R \int_0^\pi 
g(s,t,\theta)\, d\theta ds dt\]
is bounded from above and below by a constant multiple of $R^3$. We start with the upper bound.

\begin{lemma}\label{lem:UpperBoundI}
    There exists a constant $C_d > 0$ only depending on $d$, such that $I(R) \leq C_d  R^3$ for all $R \geq 1$.
\end{lemma}
\begin{proof} 
    For fixed $s,t > 0$, can use the lower bound from \eqref{eq:upper_lower_dist} to bound
    \begin{align*}
        \int_0^{e^{-s \wedge t}} g(s,t,\varphi) \,d\varphi
        &\leq e^{-\frac{d-1}{2}(s+t+|s-t|)} \sinh^{d-1}(s) \sinh^{d-1}(t) \int_0^{e^{-s\wedge t}} \sin^{d-2}(\varphi) \,d \varphi\\
        &\leq e^{-\frac{d-1}{2}(s+t+|s-t|)} e^{(d-1)(s+t)} \int_0^{e^{-s\wedge t}} \varphi^{d-2} \, d \varphi\\
        &= e^{\frac{d-1}{2}(s+t-|s-t|)} \frac{1}{d-1} (e^{-s\wedge t})^{d-1}\\
        &= \frac{1}{d-1} e^{\frac{d-1}{2}(s+t-|s-t| - 2 (s\wedge t))}\\
        &= \frac{1}{d-1},
    \end{align*}
    also using that $\sinh(x)\leq e^x$ for all $x\in\RR$.
    This yields
    \begin{equation*}
        \int_0^R \int_0^R \int_0^{e^{-(s\wedge t)}} g(s,t,\varphi) \,d \varphi ds dt \leq \frac{R^2}{d-1}. 
    \end{equation*}

    For $s,t\geq 0$ and $\varphi \in [0,\pi]$, the elementary inequalities $\arcosh(x) \geq \log(x)$, $x \geq 1$, and $\cosh(x)\geq \sinh(x)$, $x\geq 0$, yield
    \[ e^{d(s,t,\varphi)} \geq \cosh(s)\cosh(t) - \sinh(s)\sinh(t)\cos(\varphi) \geq \sinh(s) \sinh(t) (1-\cos(\varphi)),   \]
    and thus
    \begin{align*}
    g(s,t,\varphi) 
    &\leq e^{-\frac{d-1}{2}(s+t)} \left(\sinh(s)\sinh(t) (1-\cos(\varphi))\right)^{-\frac{d-1}{2}} \sin^{d-2}(\varphi) \sinh^{d-1}(s)\sinh^{d-1}(t)\\
    &= e^{-\frac{d-1}{2}(s+t)} \left(1-\cos(\varphi)\right)^{-\frac{d-1}{2}} \sin^{d-2}(\varphi) \sinh^\frac{d-1}{2}(s)\sinh^\frac{d-1}{2}(t)\\
    &\leq \left(1-\cos(\varphi)\right)^{-\frac{d-1}{2}} \sin^{d-2}(\varphi),
    \end{align*}
    where we again used the bound $\sinh(x)\leq e^x$ for $x\in\RR$, in the last line.
    A Taylor expansion of the cosine also yields $1-\cos(x) \geq x^2/8$ for $x \in [0,\pi]$.
    Combined with the above inequality and the elementary inequality $\sin(x) \leq x$, $x \geq 0$, we get
    \begin{align*}
        \int_0^R \int_0^R \int_{e^{-(s\wedge t)}}^\pi g(s,t,\varphi) \, d\varphi ds dt
        &\leq \int_0^R \int_0^R \int_{e^{-(s\wedge t)}}^\pi 8^\frac{d-1}{2} \varphi^{-(d-1)} \varphi^{d-2} \, d\varphi ds dt\\
        &= \int_0^R \int_0^R  8^\frac{d-1}{2} (s\wedge t +\log(\pi)) \,ds dt\\
        &= 8^\frac{d-1}{2} \Big(\log(\pi) R^2 + \frac{1}{3}R^3\Big).
    \end{align*}
    Combined with the bound for $\varphi \in [0,e^{-s\wedge t}]$, the claim follows.
\end{proof}

\begin{lemma}\label{lem:LowerBoundI}
    There exists a constant $c_d > 0$ only depending on $d$, such that $I(R) \geq c_d  R^3$, $R\geq 2$.
\end{lemma}
\begin{proof}
    Fix $s,t \geq 1$ and $\varphi \in [e^{-(s\wedge t)},1]$.
    Using the bound $\arcosh(x) \leq \log(2x)$, $x \geq 1$, we obtain
    \begin{align*}
    e^{d(s,t,\varphi)} 
    &\leq 2(\cosh(s)\cosh(t)-\sinh(s)\sinh(t)\cos(\varphi))\\
    &= 2(\cosh(s-t) + \sinh(s)\sinh(t)(1-\cos(\varphi)))\\
    &\leq 2(e^{|s-t|} + e^{s+t}(1-\cos(\varphi)))
    \end{align*}
    where we used the identity $\cosh(x-y) = \cosh(x)\cosh(y)-\sinh(x)\sinh(y)$, $x,y\in\RR$, in the second line, and the inequalities $\cosh(x) \leq e^{|x|}$ and $\sinh(x)\leq e^x$, for  $x \in \RR$  in the last line.
    
    Recall that $\varphi \geq e^{-(s\wedge t)}$.
    Using the bound $1-\cos(x) \geq x^2/8$, $x \in [0,\pi]$, from the previous lemma, we observe that
    \[ e^{s+t}(1-\cos(\varphi))
        \geq {1\over 8}e^{s+t} \varphi^2
        \geq {1\over 8}e^s e^t e^{-2(s\wedge t)}
        = {1\over 8}e^{|s-t|}. \]
    Together with the bound we derived above and the inequality $1-\cos(x) \leq x^2/2$, $x\geq 0$, which can again be derived using a Taylor expansion, this gives
    \[ e^{d(s,t,\varphi)}
    \leq 2(1+8) e^{s+t}(1-\cos(\varphi)) 
    \leq 9 e^{s+t} \varphi^2. \]
     We can thus bound 
    \begin{align*}
        g(s,t,\varphi)
        &\geq e^{-\frac{d-1}{2}(s+t)} (9 e^{s+t} \varphi^2)^{-\frac{d-1}{2}} \sin^{d-2}(\varphi)\sinh^{d-1}(s)\sinh^{d-1}(t)\\
        &= 9^{-\frac{d-1}{2}} e^{-(d-1)(s+t)} \sinh^{d-1}(s)\sinh^{d-1}(t) \varphi^{-(d-1)} \sin^{d-2}(\varphi).
    \end{align*}
Recall that $\sinh(x) \geq e^x/4$ for $x \geq 1$ and note that $\sin(x) \geq x/2$ for $x\in[0,1]$, so that 
    \[ g(s,t,\varphi) \geq C_1 \varphi^{-1}\, \]
    with a constant $C_1 > 0$ only depending on the dimension.
    This yields
    \begin{align*}
    I(R) 
    &\geq \int_1^R \int_1^R \int_{e^{-(s\wedge t)}}^1 g(s,t,\varphi) \,d\varphi ds dt\\
    &\geq C_1 \int_1^R \int_1^R (s \wedge t) \,d\varphi ds dt\\
    &= C_1 \left(\frac{1}{3} R^3 - R + \frac{2}{3} \right),
    \end{align*}
    and the claim follows.
\end{proof}

\begin{proof}[Proof of Theorem \ref{thm:critical}]
    The result is a direct consequence of \eqref{eq:TruncSecMomInt}, Lemma \ref{lem:UpperBoundI} and Lemma \ref{lem:LowerBoundI}.
\end{proof}

\begin{remark}
    The asymptotics for the first and second moment of $\cH^d(Z_o \cap B_R)$ as $R \to \infty$ at the critical intensity $\gamma=\gamma_c^{(d)}$ are similar to those of a critical Galton-Watson process, which may be seen as a kind of discrete analogue of the zero-cell.

    Indeed, let $X$ be a non-negative, integer-valued random variable with mean $1$ and variance $0<\sigma^2<\infty$, and consider the associated Galton-Watson process $(Z_n)_{n \in \NN_0}$ with $Z_0 \coloneqq 1$ and
    \[ Z_{n+1} \coloneqq \sum_{i=1}^{Z_n} X_{n,i}, \]
    where $(X_{n,i})_{n,i\geq 1}$ is an array of i.i.d.\ copies of $X$.
    Summing the number of individuals in each generation, we get the total offspring $V_n \coloneqq Z_0 + \ldots + Z_n$ after $n$ generations, which satisfies
 \begin{equation}\label{eq:asymptotics_GW}
        \E[V_n] = n+1 \qquad \text{and} \qquad \Var(V_n) = \frac{n(n+1)(2n+1)}{6} \sigma^2 \sim \frac{n^3}{3}\sigma^2.
    \end{equation}
    Indeed, $\E[V_n] = (n+1) \cdot \E[Z_0] = n+1$, since $(Z_n)_{n\in\NN_0}$ is a martingale according to \cite[Theorem I.6.1]{AthreyaNey}.
    Using the martingale property, we further observe that
    \[ \cov(Z_i,Z_j) = \cov(Z_i,\EE[Z_j|Z_i]) = \cov(Z_i,Z_i) \]
    for $0 \leq i \leq j$, so that
    \[ \Var(V_n) = \sum_{i=0}^n \sum_{j=0}^n \cov(Z_i,Z_j) = \sum_{i=0}^n\sum_{j=0}^n \Var(Z_{i\wedge j}). \]
    Combined with the identity $\Var(Z_i) = i \cdot \sigma^2$ from \cite[Equation (I.2.2)]{AthreyaNey}, one arrives at \eqref{eq:asymptotics_GW}.
\end{remark}

\begin{remark}
    A (central) limit theorem in the critical case cannot be expected.
    Indeed, since the zero-cell is almost surely bounded as shown in \cite{BuehlerGusakovaRecke,GodlandKabluchkoThaeleBetaStar,PorretBlanc}, $\cH^{d}(Z_0 \cap B_R) / a(R) \to 0$ almost surely, as $R\to\infty$, for any positive function $a$ with $a(R) \to \infty$.
    Hence, the only distributions that could arise as limits of $(\cH^{d}(Z_0 \cap B_R) - b(R)) / a(R)$ are degenerate ones.

    The same is true for critical Galton-Watson processes.
    Since the extinction probability of such a process is $1$ by \cite[Theorem I.5.1]{AthreyaNey}, it holds that $V_n / a_n \to 0$, $n \to \infty$, a.s.\ for any positive sequence with $a_n \to \infty$ using the same notation as in the previous remark.
    Hence, if $(V_n-b_n)/a_n$ converges in distribution, the limit must be degenerate.
\end{remark}

\section{The exact second moment: Proof of Theorem \ref{thm:supercritical} and its corollaries}\label{sec:2ndMomentExact}

\subsection{Proof of Theorem \ref{thm:supercritical}}

Consider the zero cell $Z_o$ of a hyperbolic Poisson hyperplane tessellation in the supercritical case $\gamma>\gamma_c^{(d)}$, in which it is bounded almost surely. Taking in \eqref{eq:TruncatedMomentIntegralRep} the limit as $R\to\infty$ and using the monotone convergence theorem, we obtain the identity
\begin{equation}\label{eq:2ndMomStart_d}
	\EE\cH^d(Z_o)^2
	= \int_{\HH^d}\int_{\HH^d}
	e^{-\gamma c_d(d_h(o,x)+d_h(o,y)+d_h(x,y))}\,\cH^d(dx)\cH^d(dy),
\end{equation}
with $c_d$ given by \eqref{eq:def_cd}. Let $f: \HH^d \to \RR$ be the radial function defined by
$$f(x) \coloneqq e^{-\gamma c_d d_h(o,x)}, \qquad x \in \HH^d.$$
Since
\[ d_h(x,y) = d_h(g_x \cdot o, g_y \cdot o) = d_h(o,g_x^{-1}g_y \cdot o) \]
for any $g_x,g_y \in G$ with $g_x \cdot o = x$, $g_y \cdot o = y$, the integrand in \eqref{eq:2ndMomStart_d} is equal to
$$e^{-\gamma c_d(d_h(o,x)+d_h(o,y)+d_h(x,y))} = f(x)\,f(y)\,f(x^{-1}y),$$
with $f(x^{-1}y)$ defined as in \eqref{eq:def_inverse}.
Recalling \eqref{eq:2ndMomStart_d}, the second moment can then be written as
\begin{equation}\label{eq:2ndMomConvolution}
    \EE\cH^d(Z_o)^2
	= \int_{\HH^d}\int_{\HH^d}
	f(x) f(y) f(x^{-1}y)\,\cH^d(dx)\cH^d(dy)
    = \int_{\HH^d} f(y) (f*f)(y) \cH^d(dy).
\end{equation}

The goal is to evaluate the second moment by passing to the spectral side via the spherical transform.
This follows the general strategy used in the Euclidean setting and leads to \eqref{eq:second_moment_help1}, which is the hyperbolic counterpart of \cite[Equation (16.16)]{HS}.
There is, however, a subtle difference in the argument.
Whereas in the Euclidean case the function $x\mapsto e^{-c|x|}$, $c>0$, belongs to $L^p(\RR^d)$ for every $p\in[1,\infty]$, the same is not true for our function $f$ throughout the whole supercritical regime $\gamma>\gamma_c^{(d)}$.

Indeed, let $p\in[1,\infty)$.
By passing to polar coordinates and using that $\sinh(r)$ be haves like $e^r/2$ as $r\to\infty$, it is easy to see that there exist constants $c,C>0$ such that
\[
c\int_1^\infty e^{-p\gamma c_d r+(d-1)r}\,dr
\leq
\omega_d\int_0^\infty e^{-p\gamma c_d r}\sinh^{d-1}(r)\,dr
=
\int_{\HH^d} |f(x)|^p\,\cH^d(dx)
\leq
C\int_0^\infty e^{-p\gamma c_d r+(d-1)r}\,dr.
\]
Consequently,
\[
f\in L^p(\HH^d)
\qquad\Longleftrightarrow\qquad
p>\frac{d-1}{\gamma c_d}
=
\frac{2\gamma_c^{(d)}}{\gamma}.
\]
Since $\gamma>\gamma_c^{(d)}$, the interval $(\frac{2\gamma_c^{(d)}}{\gamma},2)$
is non-empty, and hence one can choose $p\in[1,2)$ such that $f\in L^p(\HH^d)$.
The Kunze-Stein phenomenon described in Proposition \ref{prop:kunze_stein} therefore implies that
$f*f\in L^2(\HH^d)$.
In view of \eqref{eq:2ndMomConvolution}, it follows that the second moment $\E\cH^d(Z_o)^2$ can be written as the inner product
\[
\E\cH^d(Z_o)^2=\langle f,f*f\rangle_{L^2(\HH^d)}.
\]
Moreover, Lemma \ref{lem:convolution} yields $\widehat{(f*f)}=(\widehat f)^2$.
Using the Plancherel identity \eqref{eq:Plancherel}, and noting that $f$ and the spherical functions $\phi_{-\lambda}$ are real-valued for $\lambda\in\RR$, so that $\widehat f(\lambda)\in\RR$, we arrive at
\begin{equation}\label{eq:second_moment_help1}
\EE\cH^d(Z_o)^2
=
\langle f,f*f\rangle_{L^2(\HH^d)}
=
\int_0^\infty \widehat f(\lambda)\,\overline{\widehat{(f*f)}(\lambda)}\,|c(\lambda)|^{-2}\,d\lambda
=
\int_0^\infty \widehat f(\lambda)^3\,|c(\lambda)|^{-2}\,d\lambda.
\end{equation}
Substituting the value for $|c(\lambda)|^{-2}$ given by \eqref{eq:Cfunction}, we get
\begin{equation}\label{eq:2ndMomAfterSphTrafo_d}
	\EE\cH^d(Z_o)^2 = \frac{2^{1-d}\pi^{-d/2}}{\Gamma(d/2)} \int_0^\infty \widehat f(\lambda)^3 \left|\frac{\Gamma\bigl(\frac{d-1}{2}+i\lambda\bigr)}{\Gamma(i\lambda)}\right|^2\,d\lambda.
\end{equation}

Next, we compute the spherical transform of $f$ explicitly.

\begin{lemma}\label{lem:SphTrafo}
We have
\[
\widehat{f}(\lambda)
 = \frac{\gamma}{4} \pi^{\frac{d-2}{2}} \Gamma\left(\frac{d}{2}\right)
\frac{\left| \Gamma\left(\frac{\gamma c_d + 1/2 - d/2 + i\lambda}{2}\right) \right|^2}{
	\left| \Gamma\left(\frac{\gamma c_d + 3/2 + d/2 + i\lambda}{2}\right) \right|^2},\qquad \lambda>0.
\]
\end{lemma}
\begin{proof}
In polar coordinates on $\HH^d$, the volume element is $\cH^d(dx)=\sinh^{d-1}(r)\,dr\,d\sigma_{d-1}(u)$, see \eqref{eq:PolarCoordinates}. Hence
$$\widehat{f}(\lambda)=\omega_d\int_0^\infty e^{-\gamma c_d r}\,\phi_\lambda(r)\,\sinh^{d-1}(r)\,dr.$$
We next use the integral representation \cite[Formula 8.715.1]{Zwillinger} for the Legendre functions to write
$$\phi_\lambda(r) = C_1(d)\,\sinh^{2-d}(r) \int_0^r \frac{\cos(\lambda s)}{(\cosh (r)-\cosh (s))^{\frac{3-d}{2}}}\,ds,$$
where the constant $C_1(d)$ is given by
$$
C_1(d)
\coloneqq
\sqrt{\frac{2}{\pi}}\frac{K(d)}{\Gamma\!\left(\frac{d-1}{2}\right)}
=
\frac{2^{\frac{d-1}{2}}\Gamma\!\left(\frac d2\right)}{\sqrt{\pi}\,\Gamma\!\left(\frac{d-1}{2}\right)}.
$$
Substituting this into the expression for $\widehat{f}(\lambda)$, we obtain
$$\widehat{f}(\lambda) = \omega_d C_1(d)\int_0^\infty e^{-\gamma c_d r}\sinh(r) \left( \int_0^r \frac{\cos(\lambda s)}{(\cosh (r)-\cosh (s))^{\frac{3-d}{2}}}\,ds \right)dr.$$
We now justify the interchange of the order of integration by verifying absolute integrability of the corresponding double integral.
Since $|\cos(\lambda s)|\le 1$, we have
\begin{align*}
&\omega_d C_1(d)\int_0^\infty \int_0^r
e^{-\gamma c_d r}\sinh(r)
\frac{|\cos(\lambda s)|}{(\cosh(r)-\cosh(s))^{\frac{3-d}{2}}}
\,ds\,dr \\
&\le
\omega_d C_1(d)\int_0^\infty e^{-\gamma c_d r}\sinh(r)
\left(
\int_0^r \frac{1}{(\cosh(r)-\cosh(s))^{\frac{3-d}{2}}}\,ds
\right)\,dr\\
&=
\omega_d \int_0^\infty e^{-\gamma c_d r}\sinh^{d-1}(r)
\phi_0(r)\,dr
\end{align*}
Since $\gamma c_d > \gamma_c^{d} c_d = (d-1)/2$ and $\phi_0(r) \leq C (1+r)e^{-(d-1)r/2}$ for some constant $C> 0$ (see, for example, \cite[Lemma 2.4 and Remark 2.5]{ChenEtAl}), this integral converges. Hence the double integral is absolutely integrable, and Fubini's theorem applies. We may therefore interchange the order of integration and obtain
\[
\widehat{f}(\lambda)
=
\omega_d C_1(d)\int_0^\infty \cos(\lambda s)\,\cA(s)\,ds,
\]
where
\[
\cA(s)\coloneqq
\int_s^\infty
\frac{e^{-\gamma c_d r}\sinh(r)}{(\cosh(r)-\cosh(s))^{\frac{3-d}{2}}}\,dr,
\qquad s\ge 0.
\]

To determine $\cA(s)$, we utilize the integral representation \cite[Formula 8.715.2]{Zwillinger} of the Legendre functions of the second kind $Q_\nu^\mu$. For $\alpha>0$, ${\rm Re}(\mu) < 1/2$ and $\nu$ satisfying ${\rm Re}(\nu+\mu)>-1$, it states that
\begin{equation}\label{eq:Q_integral_d}
	\int_\alpha^\infty \frac{e^{-(\nu+\frac12)t}}{(\cosh (t)-\cosh(\alpha))^{\mu+\frac{1}{2}}}\,dt
	=
	\sqrt{\frac{2}{\pi}} e^{-\mu\pi i}\,\Gamma\!\left(\frac{1}{2}-\mu\right)\sinh^{-\mu}(\alpha)\,
	Q_\nu^{\frac{2-d}{2}}(\cosh\alpha).
\end{equation}
Using $\sinh r=\frac12(e^r-e^{-r})$, we can rewrite $\cA(s)$ as
\[
\cA(s)
=
\frac12\int_s^\infty
\frac{e^{-(\gamma c_d-1)r}}{(\cosh (r)-\cosh (s))^{\frac{3-d}{2}}}\,dr
- \frac12\int_s^\infty
\frac{e^{-(\gamma c_d+1)r}}{(\cosh (r)-\cosh (s))^{\frac{3-d}{2}}}\,dr.
\]
We now apply \eqref{eq:Q_integral_d} to the first part with $\mu = (2-d)/2$ and $\nu = \gamma c_d - 3/2$, and to the second part with $\mu = (2-d)/2$ and $\nu = \gamma c_d - 1/2$.
Note that the condition ${\rm Re}(\nu+\mu)>-1$ is met since $\gamma c_d > (d-1)/2$.
This yields, for $s>0$,
\[
\cA(s)
=
\frac{1}{\sqrt{2\pi}}e^{\frac{d-2}{2}\pi i}\,\Gamma\!\left(\frac{d-1}{2}\right)\sinh^{\frac{d-2}{2}}(s)
\left(
Q_{\gamma c_d-\frac32}^{\frac{2-d}{2}}(\cosh (s))
-
Q_{\gamma c_d+\frac12}^{\frac{2-d}{2}}(\cosh (s))
\right).
\]

Substituting this into the representation of $\widehat{f}(\lambda)$ and collecting constants, define
\[
C_2(d)
\coloneqq\omega_d C_1(d)\frac{1}{\sqrt{2\pi}}\Gamma\!\left(\frac{d-1}{2}\right)e^{\frac{d-2}{2}\pi i}
=2(2\pi)^{\frac{d-2}{2}}e^{\frac{d-2}{2}\pi i}.
\]
Then
\begin{align*}
	\widehat{f}(\lambda)
	&= C_2(d) \left( \int_0^\infty \sinh^{\frac{d-2}{2}}(s) \, Q_{\gamma c_d - \frac 32}^{\frac{2-d}{2}}(\cosh(s)) \cos(\lambda s) \, ds\right.\\
	&\qquad\qquad\qquad\qquad \left.- \int_0^\infty \sinh^{\frac{d-2}{2}}(s) \, Q_{\gamma c_d+\frac12}^{\frac{2-d}{2}}(\cosh(s)) \cos(\lambda s) \, ds \right)\\
	&=: C_2(d) (K_1(\lambda)-K_2(\lambda)).    
\end{align*}
Hence, the spherical transform is expressed as a difference of two Fourier cosine transforms of $\sinh$-weighted Legendre functions of the second kind with degree $\gamma c_d -\frac{3}{2}$ and $\gamma c_d +\frac{1}{2}$ respectively.
We evaluate these integrals using  \cite[Entry 1.12.47]{Oberhettinger}, which states for complex numbers $\mu,\nu$ with $\re(\mu)<\frac{1}{2}$ and $\re(\nu+\mu)>-1$ that
\begin{equation}\label{eq:cosine_transform}
\int_0^\infty \sinh^{-\mu}(s) Q_\nu^\mu(\cosh(s)) \cos(\lambda s) \,ds
= C(d,\mu) \frac{
\Gamma\left(\frac{1+\nu+\mu+i\lambda}{2}\right)
\Gamma\left(\frac{1+\nu+\mu-i\lambda}{2}\right)
}{
\Gamma\left(\frac{2+\nu-\mu+i\lambda}{2}\right)
\Gamma\left(\frac{2+\nu-\mu-i\lambda}{2}\right)},
\end{equation}
with the constant $C(d,\mu)$ given by $C(d,\mu) = e^{i\pi\mu} \sqrt{\pi} 2^{\mu - 2} \Gamma\left(\frac{1}{2}-\mu\right)$. 

Applying \eqref{eq:cosine_transform} with $\mu = \frac{2-d}{2}$ and $\nu = \gamma c_d - \frac{3}{2}$ (recall that $\gamma c_d > \frac{d-1}{2}$, so $\re(\nu+\mu)>-1$ is satisfied), and using the conjugation property $\Gamma(\bar z)=\overline{\Gamma(z)}$, we obtain
\[
K_1(\lambda) = C_3(d) \frac{\left| \Gamma\left(\frac{\gamma c_d + 1/2 - d/2 + i\lambda}{2}\right) \right|^2}{\left| \Gamma\left(\frac{\gamma c_d -1/2 + d/2 + i\lambda}{2}\right) \right|^2},\]
with
\[ C_3(d) \coloneqq e^{i\pi\frac{2-d}{2}} \sqrt{\pi} 2^{\frac{-d-2}{2}} \Gamma\left( \frac{d-1}{2} \right). \]
Similarly, we get
\[
K_2(\lambda) = C_3(d) \frac{\left| \Gamma\left(\frac{\gamma c_d + 1/2 - d/2 + i\lambda}{2} +1 \right) \right|^2}{\left| \Gamma\left(\frac{\gamma c_d -1/2 + d/2 + i\lambda}{2} +1\right) \right|^2}.
\]
To unify these terms, we define the complex numbers 
\[
v \coloneqq \frac{\gamma c_d + 1/2 - d/2 + i\lambda}{2} \qquad \text{and} \qquad w \coloneqq \frac{\gamma c_d -1/2 + d/2 + i\lambda}{2}.
\]
Using the identity $\Gamma(z+1) = z \Gamma(z)$, we then obtain
\[ K_1(\lambda) - K_2(\lambda) 
= C_3(d) \left( \frac{|\Gamma(v)|^2}{|\Gamma(w)|^2} - \frac{|\Gamma(v+1)|^2}{|\Gamma(w+1)|^2} \right) 
= C_3(d) \frac{|\Gamma(v)|^2}{|\Gamma(w+1)|^2} \left(|w|^2 - |v|^2\right). \]
Since
\[ |w|^2 - |v|^2 = \gamma c_d\frac{d-1}{2}, \]
this gives
\[
\widehat{f}(\lambda) = C_2(d) C_3(d) \gamma c_d\frac{d-1}{2} \frac{\left|\Gamma\left(\frac{\gamma c_d + 1/2 -d/2 +i\lambda}{2}\right)\right|^2}{\left|\Gamma\left(\frac{\gamma c_d + 3/2 +d/2 +i\lambda}{2}\right)\right|^2}.
\]
The claim follows after simplifying the constant.
\end{proof}

We are now ready to prove Theorem \ref{thm:supercritical}.

\begin{proof}[Proof of Theorem \ref{thm:supercritical}]
In terms of the parameter $a$, the spherical transform of $f$ determined in Lemma \ref{lem:SphTrafo} reads
\begin{equation}\label{eq:SphericalTrafoFinal_d}
	\widehat{f}(\lambda) = \frac{\gamma}{4} \pi^{\frac{d-2}{2}} \Gamma\left(\frac{d}{2}\right)
	\frac{\left| \Gamma\left(a + i\frac{\lambda}{2}\right) \right|^2}{
		\left| \Gamma\left(a + \frac{d+1}{2} + i\frac{\lambda}{2}\right) \right|^2}.    
\end{equation}
Inserting \eqref{eq:SphericalTrafoFinal_d} into \eqref{eq:2ndMomAfterSphTrafo_d}, we obtain 
\begin{equation}\label{eq:SecondMomentIntermediate_d}
	\EE\cH^d(Z_o)^2
	= \frac{2^{1-d} \pi^{-\frac{d}{2}}}{\Gamma\left(\frac{d}{2}\right)} \left( \frac{\gamma}{4} \pi^{\frac{d-2}{2}} \Gamma\left(\frac{d}{2}\right) \right)^3
	\int_0^\infty
	\frac{\left| \Gamma\left(a + i\frac{\lambda}{2}\right) \right|^6}{
		\left| \Gamma\left(a + \frac{d+1}{2} + i\frac{\lambda}{2}\right) \right|^6}
	\frac{\left| \Gamma\left(\frac{d-1}{2} + i\lambda\right) \right|^2}{\left| \Gamma(i\lambda) \right|^2}
	\, d\lambda.
\end{equation}
To unify the arguments of the Gamma functions, we substitute $\lambda = 2x$ and the integral becomes
\[
\EE\cH^d(Z_o)^2
= \frac{2^{2-d}}{\pi^{d/2}} \left( \frac{\gamma}{4} \right)^3 \pi^{\frac{3d-6}{2}} \Gamma\left(\frac{d}{2}\right)^2
\int_0^\infty
\frac{\left| \Gamma(a + ix) \right|^6}{
	\left| \Gamma\left(a + \frac{d+1}{2} + ix\right) \right|^6}
\frac{\left| \Gamma\left(\frac{d-1}{2} + 2ix\right) \right|^2}{\left| \Gamma(2ix) \right|^2}
\, dx.
\]
We next simplify the integrand using the Legendre duplication formula \eqref{eq:Legendre}. For the denominator, setting $z=ix$, we obtain
\[
|\Gamma(2ix)|^2 = \frac{1}{4\pi} |\Gamma(ix)|^2 \left|\Gamma\left(ix+\frac{1}{2}\right)\right|^2.
\]
For the numerator, setting $z=ix+\frac{d-1}{4}$, we obtain
\[
\left|\Gamma\left(\frac{d-1}{2}+2ix\right)\right|^2
= \frac{2^{d-3}}{\pi} \left|\Gamma\left(ix+\frac{d-1}{4}\right)\right|^2 \left|\Gamma\left(ix+\frac{d+1}{4}\right)\right|^2.
\]
Hence, their ratio simplifies to
\[
\frac{\left| \Gamma\left(\frac{d-1}{2} + 2ix\right) \right|^2}{\left| \Gamma(2ix) \right|^2}
= 2^{d-1} \frac{\left| \Gamma\left(ix+\frac{d-1}{4}\right) \right|^2 \left| \Gamma\left(ix+\frac{d+1}{4}\right) \right|^2}{\left| \Gamma(ix) \right|^2 \left| \Gamma\left(ix+\frac{1}{2}\right) \right|^2}.
\]
Substituting this back into the last expression for $\EE\cH^d(Z_o)^2$, we arrive at
\[
\EE\cH^d(Z_o)^2
= 2 \left( \frac{\gamma}{4} \right)^3 \pi^{d-3} \Gamma\left(\frac{d}{2}\right)^2
\int_0^\infty
\frac{\left| \Gamma(a + ix) \right|^6 \left| \Gamma\left(ix+\frac{d-1}{4}\right) \right|^2 \left| \Gamma\left(ix+\frac{d+1}{4}\right) \right|^2}{
	\left| \Gamma\left(a + \frac{d+1}{2} + ix\right) \right|^6 \left| \Gamma(ix) \right|^2 \left| \Gamma\left(ix+\frac{1}{2}\right) \right|^2}
\, dx.
\]

In a next step, we convert the integral over the positive real axis to a contour integral of Mellin--Barnes type. Since the integrand is an even function of $x$, as it depends only on moduli of the form $|\Gamma(A+ix)|^2=\Gamma(A+ix)\Gamma(A-ix)$, we may extend the integration range to $(-\infty,\infty)$ and divide by $2$. Thus, writing the last display as
\[
\EE\cH^d(Z_o)^2
= \pi^{d-3}\Gamma\left(\frac{d}{2}\right)^2\left(\frac{\gamma}{4}\right)^3
\int_{-\infty}^{\infty}
\frac{\left| \Gamma(a + ix) \right|^6 \left| \Gamma\!\left(ix+\frac{d-1}{4}\right) \right|^2 \left| \Gamma\!\left(ix+\frac{d+1}{4}\right) \right|^2}{
	\left| \Gamma\!\left(a + \frac{d+1}{2} + ix\right) \right|^6 \left| \Gamma(ix) \right|^2 \left| \Gamma\!\left(ix+\frac{1}{2}\right) \right|^2}
\, dx,
\]
we set $s=ix$, so that $dx=\frac{1}{i}\,ds$ and the path of integration becomes the imaginary axis $i\RR$. Multiplying and dividing by $2\pi$ to match standard contour integral notation, and expanding the squared moduli, yields
\begin{align}\label{eq:SecondMomentMellinBarnes_d}
	\nonumber\EE\cH^d(Z_o)^2
	&= 2 \pi^{d-2}\Gamma\left(\frac{d}{2}\right)^2 \left( \frac{\gamma}{4} \right)^3 \\
	&\quad\times\frac{1}{2\pi i}
	\int_{-i\infty}^{i\infty}
	\frac{\Gamma(a+s)^3 \Gamma(a-s)^3 \Gamma(\frac{d-1}{4}+s)\Gamma(\frac{d-1}{4}-s)\Gamma(\frac{d+1}{4}+s)\Gamma(\frac{d+1}{4}-s)}{
		\Gamma(a+\frac{d+1}{2}+s)^3 \Gamma(a+\frac{d+1}{2}-s)^3 \Gamma(s)\Gamma(-s)\Gamma(\frac{1}{2}+s)\Gamma(\frac{1}{2}-s)}
	\, ds.
\end{align}

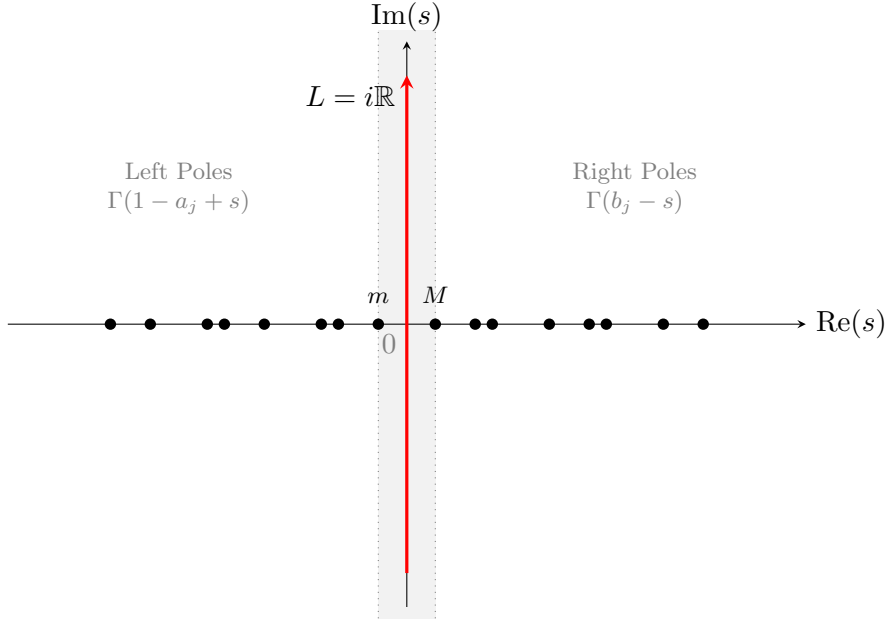
\begin{figure}[t]
    \centering
    \begin{tikzpicture}[scale=1.5, >=stealth]
        \def\a{0.6}       
        \def\eps{0}       
        \def\dotSize{0.05} 

        \fill[gray!10] (-0.25, -2.6) rectangle (0.25, 2.6);
        \draw[gray, dotted] (-0.25, -2.6) -- (-0.25, 2.6);
        \draw[gray, dotted] (0.25, -2.6) -- (0.25, 2.6);

        \draw[->] (-3.5, 0) -- (3.5, 0) node[right] {$\mathrm{Re}(s)$};
        \draw[->] (0, -2.5) -- (0, 2.5) node[above] {$\mathrm{Im}(s)$};

        \draw[very thick, red, ->] (0, -2.2) -- (0, 2.2) node[below left, black] {$L=i\RR$};

        \foreach \k in {0, 1, 2} { \fill[black] (-\a-\k, 0) circle (\dotSize); }
        \foreach \x in {-0.25, -0.75, -1.25, -1.75, -2.25} { \fill[black] (\x, 0) circle (\dotSize); }
        \node[above, font=\footnotesize] at (-0.25, 0.1) {$m$};

        \foreach \k in {0, 1, 2} { \fill[black] (\a+\k, 0) circle (\dotSize); }
        \foreach \x in {0.25, 0.75, 1.25, 1.75, 2.25} { \fill[black] (\x, 0) circle (\dotSize); }
        \node[above, font=\footnotesize] at (0.25, 0.1) {$M$};

        \node[font=\footnotesize, align=center, gray] at (-2, 1.2) {Left Poles\\ $\Gamma(1-a_j+s)$};
        \node[font=\footnotesize, align=center, gray] at (2, 1.2) {Right Poles\\ $\Gamma(b_j-s)$};
        \node[below left, gray] at (0,0) {$0$};
    \end{tikzpicture}
    \caption{The path of integration $L = i\RR$ separates the left pole families from the right pole families.}
    \label{fig:ContourGFunction}
\end{figure}

Let $h(a;s)$ be the integrand in \eqref{eq:SecondMomentMellinBarnes_d}. Since the Gamma function is meromorphic with simple poles at $0,-1,-2,\ldots$ and has no zeros, the reciprocal $1/\Gamma(\cdot)$ is an entire function. Consequently, possible poles of $h(a;s)$ can only come from the Gamma factors in the numerator. These occur at
\[
s=-a-n,\qquad s=a+n,\qquad s=\pm\Bigl(\tfrac{d-1}{4}+n\Bigr),\qquad s=\pm\Bigl(\tfrac{d+1}{4}+n\Bigr),
\qquad n\in\{0,1,2,\ldots\}.
\]
In the supercritical regime $\gamma>\gamma_c^{(d)}$, we have $\gamma c_d > \frac{d-1}{2}$. Consequently, the parameter $a = \frac{\gamma c_d}{2} + \frac{1-d}{4}$ satisfies $a > \frac{d-1}{4} + \frac{1-d}{4} = 0$. Because both $a > 0$ and $\frac{d-1}{4} > 0$ for $d \ge 2$, the poles are strictly separated by the imaginary axis:
\begin{itemize}
	\item The left poles lie in the half-plane $\operatorname{Re}(s)\le m$, where
	$m\coloneqq\max\{-a,-\tfrac{d-1}{4}\}<0$.
	\item The right poles lie in the half-plane $\operatorname{Re}(s)\ge M$, where
	$M\coloneqq\min\{a,\tfrac{d-1}{4}\}>0$.
\end{itemize}
Thus, the integrand $h(a;s)$ is holomorphic on the vertical strip $\{s\in\CC:|\operatorname{Re}(s)|<\min\{a,\frac{d-1}{4}\}\}$ containing the imaginary axis, and the contour integral in \eqref{eq:SecondMomentMellinBarnes_d} is well-defined along the path $L=i\RR$, see Figure \ref{fig:ContourGFunction}.

The integral in \eqref{eq:SecondMomentMellinBarnes_d} is a particular instance of the Meijer $G$-function. Generally, for integers $0 \le m \le q$ and $0 \le n \le p$, and complex parameters $a_1, \dots, a_p$ and $b_1, \dots, b_q$, the $G$-function is defined by the Mellin--Barnes type integral 
\begin{equation}\label{eq:DefGFunction}
	G_{p,q}^{m,n} \left( z \middle| \begin{matrix} a_1, \dots, a_p \\ b_1, \dots, b_q \end{matrix} \right)
	\coloneqq \frac{1}{2\pi i} \int_L
	\frac{\prod_{j=1}^m \Gamma(b_j - s)\prod_{j=1}^n \Gamma(1-a_j+s)}{
		\prod_{j=m+1}^q \Gamma(1-b_j+s)\prod_{j=n+1}^p \Gamma(a_j-s)}
	\, z^s \, ds,
\end{equation}
see \cite[Equation (1.1.1)]{MathaiSaxena}. The integration contour $L$ is chosen to separate the poles of the factors $\Gamma(b_j-s)$ for $j=1,\dots,m$ from those of the factors $\Gamma(1-a_j+s)$ for $j=1,\dots,n$, see \cite[page 2]{MathaiSaxena}. 

Comparing \eqref{eq:SecondMomentMellinBarnes_d} with \eqref{eq:DefGFunction} for $z=1$ and contour $L=i\RR$, we identify:
\begin{itemize}
	\item The dimensions are $m=5$, $n=5$, $p=10$, and $q=10$.
	\item Right poles ($b_j$): The factors $\Gamma(a-s)^3, \Gamma(\frac{d-1}{4}-s), \Gamma(\frac{d+1}{4}-s)$ yield
	\[ b_1=b_2=b_3=a,\quad b_4=\tfrac{d-1}{4},\quad b_5=\tfrac{d+1}{4}. \]
	\item Left poles ($a_j$): The factors $\Gamma(a+s)^3, \Gamma(\frac{d-1}{4}+s), \Gamma(\frac{d+1}{4}+s)$ yield
	\[ a_1=a_2=a_3=1-a,\quad a_4=\tfrac{5-d}{4},\quad a_5=\tfrac{3-d}{4}. \]
	\item Denominator parameters: The factors $\Gamma(a+\frac{d+1}{2}\pm s)^3$, $\Gamma(s)\Gamma(-s)$, and
	$\Gamma(\frac12+s)\Gamma(\frac12-s)$ correspond to
	\[
	a_6=a_7=a_8=a+\tfrac{d+1}{2},\quad a_9=0,\quad a_{10}=\tfrac12,
	\qquad
	b_6=b_7=b_8=\tfrac{1-d}{2}-a,\quad b_9=1,\quad b_{10}=\tfrac12 .
	\]
\end{itemize}
Since in the supercritical regime $\gamma>\gamma_c^{(d)}$, the parameter $a$ is strictly positive, ensuring that the defining conditions of the Meijer $G$-function are satisfied by the contour $L=i\RR$:
\begin{enumerate}
	\item Separation of poles: As established above, since $a>0$, the contour $L=i\RR$ strictly separates the left pole set from the right pole set.
	\item Convergence at $z=1$: Since $p=q=10$ and $m=n=5$, we are in the case $m+n=p$, where the integral converges absolutely on the vertical line if and only if the parameter summation satisfies $\operatorname{Re}(\sum_{j=1}^q b_j - \sum_{j=1}^p a_j) < -1$, see \cite[Section 1.1]{MathaiSaxena}. Evaluating the sums for our parameters, we find
	\[
	\sum_{j=1}^{10} b_j = 3-d \qquad \text{and} \qquad \sum_{j=1}^{10} a_j = d+7.
	\]
	The difference is $(3-d) - (d+7) = -2d - 4$. For any dimension $d \ge 2$, this difference is bounded above by $-8$, which is strictly less than $-1$. Thus, the integral converges absolutely for all dimensions.
\end{enumerate}
Collecting these parameters and including the prefactor, we have thus completed the proof.
\end{proof}

\subsection{Proofs of Corollaries \ref{cor:odd_dimensions}-\ref{cor:GammaToCritical_d}}

\begin{proof}[Proof of Corollary \ref{cor:odd_dimensions}]
	Starting from \eqref{eq:SecondMomentMellinBarnes_d}, we substitute \(d=2k+1\). Then
	\[
	\frac{d+1}{2}=k+1,
	\qquad
	\frac{d-1}{4}=\frac{k}{2},
	\qquad
	\frac{d+1}{4}=\frac{k+1}{2},
	\]
	and the parameter \(a\) becomes
	\[
	a=\frac{\gamma c_{2k+1}}{2}+\frac{1-(2k+1)}{4}
	=\frac{\gamma c_{2k+1}}{2}-\frac{k}{2}.
	\]
	Hence the Mellin--Barnes integrand is
	\[
	\frac{\Gamma(a+s)^3\Gamma(a-s)^3
		\Gamma\!\left(\frac{k}{2}+s\right)\Gamma\!\left(\frac{k+1}{2}+s\right)
		\Gamma\!\left(\frac{k}{2}-s\right)\Gamma\!\left(\frac{k+1}{2}-s\right)}
	{\Gamma(a+k+1+s)^3\Gamma(a+k+1-s)^3
		\Gamma(s)\Gamma(-s)\Gamma\!\left(\frac12+s\right)\Gamma\!\left(\frac12-s\right)}.
	\]
	
	We simplify the \(a\)-dependent Gamma factors using
	\(\Gamma(z+n)=\Gamma(z)\prod_{r=0}^{n-1}(z+r)\) with \(n=k+1\):
	\[
	\frac{\Gamma(a+s)^3\Gamma(a-s)^3}
	{\Gamma(a+k+1+s)^3\Gamma(a+k+1-s)^3}
	=
	\frac{1}{\prod_{m=0}^{k}(a+m+s)^3(a+m-s)^3}
	=
	\frac{1}{\prod_{m=0}^{k}\bigl((a+m)^2-s^2\bigr)^3}.
	\]
	
	Next, for the dimension-dependent part, apply the Legendre duplication formula \eqref{eq:Legendre}
	with \(z=\frac{k}{2}+s\) and \(z=\frac{k}{2}-s\), and similarly in the denominator with
	\(z=s\) and \(z=-s\). This gives
	\[
	\frac{\Gamma\!\left(\frac{k}{2}+s\right)\Gamma\!\left(\frac{k+1}{2}+s\right)
		\Gamma\!\left(\frac{k}{2}-s\right)\Gamma\!\left(\frac{k+1}{2}-s\right)}
	{\Gamma(s)\Gamma\!\left(\frac12+s\right)\Gamma(-s)\Gamma\!\left(\frac12-s\right)}
	=
	2^{-2k}\frac{\Gamma(k+2s)\Gamma(k-2s)}{\Gamma(2s)\Gamma(-2s)}.
	\]
	Since \(k\in\mathbb{N}\), repeated use of \(\Gamma(z+n)=\Gamma(z)\prod_{j=0}^{n-1}(z+j)\) yields
	\[
	\frac{\Gamma(k+2s)}{\Gamma(2s)}=\prod_{j=0}^{k-1}(2s+j),
	\qquad
	\frac{\Gamma(k-2s)}{\Gamma(-2s)}=\prod_{j=0}^{k-1}(-2s+j),
	\]
	and therefore
	\[
	2^{-2k}\frac{\Gamma(k+2s)\Gamma(k-2s)}{\Gamma(2s)\Gamma(-2s)}
	=
	2^{-2k}\prod_{j=0}^{k-1}(j+2s)(j-2s)
	=
	\prod_{j=0}^{k-1}\left(\frac{j^2}{4}-s^2\right)
	=P_k(s).
	\]
	
	Thus the Mellin--Barnes integral in the representation of $\EE\cH^{2k+1}(Z_o)^2$ reduces to
	\[
	\frac{1}{2\pi i}\int_{-i\infty}^{i\infty}\frac{P_k(s)}{Q_k(s)}\,ds.
	\]
	Combining this with the prefactor in \eqref{eq:SecondMomentMellinBarnes_d} (specialized to \(d=2k+1\))
	gives
	\[
	\EE \cH^{2k+1}(Z_o)^2
	=
	2\pi^{2k-1}\Gamma\!\left(k+\frac12\right)^2\left(\frac{\gamma}{4}\right)^3
	\cdot
	\frac{1}{2\pi i}\int_{-i\infty}^{i\infty}\frac{P_k(s)}{Q_k(s)}\,ds.
	\]
	
	Since \(\gamma>\gamma_c^{(2k+1)}\), we have \(a>0\), so the poles
	\[
	s=-(a+m),\qquad m=0,\dots,k,
	\]
	lie in the open left half-plane. Moreover, the integrand is rational and satisfies
	\[
	\frac{P_k(s)}{Q_k(s)} \leq C|s|^{-4k-6}\qquad \text{as}\;\;|s|\to\infty
	\]
    for some constant $C>0$.
	By the residue theorem, closing the contour to the left by a semicircle yields
	\[
	\frac{1}{2\pi i}\int_{-i\infty}^{i\infty}\frac{P_k(s)}{Q_k(s)}\,ds
	=
	\sum_{m=0}^{k}\Res\limits_{s=-(a+m)}\frac{P_k(s)}{Q_k(s)}.
	\]
	Substituting this into the previous identity proves the claim.
\end{proof}

\begin{proof}[Proof of Corollary \ref{cor:GammaToInfinity_d}]
    We start from the integral representation \eqref{eq:SecondMomentIntermediate_d},
	\[
	\EE\cH^d(Z_o)^2
	=
	K(d,\gamma)
	\int_0^\infty
	R_a(\lambda)\,P_d(\lambda)\,d\lambda,
	\]
	where
	\[
	R_a(\lambda)\coloneqq
	\frac{\left| \Gamma\left(a + i\frac{\lambda}{2}\right) \right|^6}{
		\left| \Gamma\left(a + \frac{d+1}{2} + i\frac{\lambda}{2}\right) \right|^6},
	\qquad
	P_d(\lambda)\coloneqq
	\frac{\left| \Gamma\left(\frac{d-1}{2} + i\lambda\right) \right|^2}{
		\left| \Gamma(i\lambda) \right|^2},
	\]
	and
	\[
	a=\frac{\gamma c_d}{2}+\frac{1-d}{4},
	\qquad
	K(d,\gamma)=\frac{\gamma^3}{2^{d+5}}\pi^{d-3}\Gamma\!\left(\frac d2\right)^2.
	\]
	As $\gamma\to\infty$, we have $a\to\infty$ and
	\[
	a\sim \frac{\gamma c_d}{2}
	=
	\frac{\gamma\,\Gamma\!\left(\frac d2\right)}{4\sqrt{\pi}\,\Gamma\!\left(\frac{d+1}{2}\right)}.
	\]

	To obtain the asymptotics of the integral, we put $\lambda=2au$.
	Then
	\[
	a^{2d+3}\int_0^\infty R_a(\lambda)P_d(\lambda)\,d\lambda
	=
	\int_0^\infty f_a(u)\,du,
	\]
	where $f_a(u)\coloneqq 2a^{2d+4}R_a(2au)P_d(2au)$ for $u\geq 0$.

	We first determine the pointwise limit of $f_a(u)$. By the asymptotic expansion
	\cite[Formula 5.11.13]{NIST},
    \begin{equation}\label{eq:gamma_ratio}
    \frac{\Gamma(z+\alpha)}{\Gamma(z+\beta)}
	=
	z^{\alpha-\beta}(1+O(1/z)),
	\qquad |z|\to\infty,
    \end{equation}
	uniformly in sectors $|\arg z|\le \pi-\delta$, $\delta > 0$.
    We apply this with
	\[
	z=a(1+iu),\qquad \alpha=0,\qquad \beta=\frac{d+1}{2}.
	\]
	Since $a>0$ and $\arg(1+iu)\in[-\frac{\pi}{2},\frac{\pi}{2}]$, the sector condition is satisfied uniformly in $u\ge 0$. Hence, for every fixed $u\ge 0$,
	\[
	R_a(2au)
	=
	\frac{|\Gamma(a(1+iu))|^6}{|\Gamma(a(1+iu)+\frac{d+1}{2})|^6}
	\sim
	a^{-3(d+1)}|1+iu|^{-3(d+1)}
	=
	a^{-3(d+1)}(1+u^2)^{-\frac{3(d+1)}{2}}.
	\]
	Moreover, for every fixed $u>0$, another application of the same asymptotic formula with
	$z=2iau$, $\alpha=\frac{d-1}{2}$ and $\beta=0$ yields
	\[
	P_d(2au)
	=
	\frac{\left| \Gamma\left(\frac{d-1}{2} + 2iau\right) \right|^2}{
		\left| \Gamma(2iau) \right|^2}
	\sim
	|2iau|^{d-1}
	=
	(2au)^{d-1}.
	\]
	Therefore, for every fixed $u>0$,
	\[
	f_a(u)
	\to
	2^d\,u^{d-1}(1+u^2)^{-\frac{3(d+1)}{2}},
	\qquad\text{as } a\to\infty.
	\]
	For $u=0$ we have $f_a(0)=0$ for all $a$, so the same limit holds on all of $[0,\infty)$.

	Next we establish an integrable majorant. Since the above Gamma-ratio estimate is uniform in sectors, there exist constants $C_1>0$ and $a_0>0$ such that for all $a\ge a_0$ and all $u\ge 0$,
	\[
	R_a(2au)\le C_1 a^{-3(d+1)}(1+u^2)^{-\frac{3(d+1)}{2}}.
	\]
	Furthermore, there exists a constant $C_2>0$ such that
	\[
	0\le P_d(\lambda)\le C_2(1+\lambda^{d-1}),
	\qquad \lambda\ge 0.
	\]
	Indeed, for $\lambda\ge 1$ this follows from Stirling's formula or by the Gamma ratio bound above, while on $[0,1]$ the function
	$\lambda\mapsto P_d(\lambda)$ is continuous and bounded, with $P_d(0)=0$.
	Hence, for all $a\ge a_0$ and all $u\ge 0$,
	\begin{align*}
		f_a(u)
		&\le
		2C_1a^{2d+4}a^{-3(d+1)}(1+u^2)^{-\frac{3(d+1)}{2}}\,C_2(1+(2au)^{d-1}) \\
		&\le
		C( a^{1-d}+u^{d-1})(1+u^2)^{-\frac{3(d+1)}{2}} \\
		&\le
		C'(1+u^{d-1})(1+u^2)^{-\frac{3(d+1)}{2}},
	\end{align*}
	where $C,C'>0$ are constants independent of $a$ and $u$. Since
	$(1+u^{d-1})(1+u^2)^{-\frac{3(d+1)}{2}}$
	is integrable on $[0,\infty)$, the dominated convergence theorem yields
	\[
	a^{2d+3}\int_0^\infty R_a(\lambda)P_d(\lambda)\,d\lambda
	\longrightarrow
	2^d\int_0^\infty u^{d-1}(1+u^2)^{-\frac{3(d+1)}{2}}\,du.
	\]

	It remains to compute the limit integral. Using the substitution $t=u^2$, so that
	$du=\frac12 t^{-1/2}\,dt$, we obtain
	\begin{align*}
		2^d\int_0^\infty u^{d-1}(1+u^2)^{-\frac{3(d+1)}{2}}\,du
		&=
		2^{d-1}\int_0^\infty t^{\frac d2-1}(1+t)^{-\frac{3d+3}{2}}\,dt \\
		&=
		2^{d-1}B\!\left(\frac d2,d+\frac32\right) \\
		&=
		2^{d-1}\frac{\Gamma\!\left(\frac d2\right)\Gamma\!\left(d+\frac32\right)}{
			\Gamma\!\left(\frac{3d+3}{2}\right)},
	\end{align*}
    where $B(\,\cdot\,,\,\cdot\,)$ is the Beta function.
	Consequently,
	\[
	\int_0^\infty R_a(\lambda)P_d(\lambda)\,d\lambda
	\sim
	2^{d-1}\frac{\Gamma\!\left(\frac d2\right)\Gamma\!\left(d+\frac32\right)}{
		\Gamma\!\left(\frac{3d+3}{2}\right)}
	a^{-2d-3}.
	\]
	Combining this with the prefactor $K(d,\gamma)$ and using the asymptotics $a\sim{\gamma c_d\over 2}$,
	we conclude that
	\begin{align*}
		\EE\cH^d(Z_o)^2
		&\sim
		\left(\frac{\gamma^3}{2^{d+5}}\pi^{d-3}\Gamma\!\left(\frac d2\right)^2\right)
		\left(
		2^{d-1}\frac{\Gamma\!\left(\frac d2\right)\Gamma\!\left(d+\frac32\right)}{
			\Gamma\!\left(\frac{3d+3}{2}\right)}
		\right)
		\left(
		\frac{\gamma\,\Gamma\!\left(\frac d2\right)}{
			4\sqrt{\pi}\,\Gamma\!\left(\frac{d+1}{2}\right)}
		\right)^{-2d-3} \\
		&=
		\gamma^{-2d}\,
		2^{4d}\pi^{\frac{4d-3}{2}}
		\frac{\Gamma\!\left(d+\frac32\right)}{\Gamma\!\left(\frac{3d+3}{2}\right)}
		\frac{\Gamma\!\left(\frac{d+1}{2}\right)^{2d+3}}{
			\Gamma\!\left(\frac d2\right)^{2d}}.
	\end{align*}
	This proves the claim.
\end{proof}

\begin{proof}[Proof of Corollary \ref{cor:GammaToCritical_d}]
	 We start from \eqref{eq:SecondMomentIntermediate_d}, which says that
	\[
	\EE\cH^d(Z_o)^2
	=
	K(d,\gamma)\int_0^\infty R_a(\lambda)\,P_d(\lambda)\,d\lambda,
	\]
	where
	\[
	R_a(\lambda)\coloneqq
	\frac{\left| \Gamma\left(a + i\frac{\lambda}{2}\right) \right|^6}{
		\left| \Gamma\left(a + \frac{d+1}{2} + i\frac{\lambda}{2}\right) \right|^6},
	\qquad
	P_d(\lambda)\coloneqq
	\frac{\left| \Gamma\left(\frac{d-1}{2} + i\lambda\right) \right|^2}{
		\left| \Gamma(i\lambda) \right|^2},
	\]
	and
	\[
	a=\frac{\gamma c_d}{2}+\frac{1-d}{4}
	=\frac{c_d}{2}\bigl(\gamma-\gamma_c^{(d)}\bigr),
	\qquad
	K(d,\gamma)=\frac{\gamma^3\pi^{d-3}}{2^{d+5}}\Gamma\!\left(\frac d2\right)^2.
	\]
	Hence \(\gamma\downarrow \gamma_c^{(d)}\) is equivalent to \(a\downarrow 0\).

	We decompose
	\[
	I(a)\coloneqq\int_0^\infty R_a(\lambda)P_d(\lambda)\,d\lambda
	= I_1(a)+I_2(a),
	\]
	where
	\[
	I_1(a)\coloneqq\int_0^1 R_a(\lambda)P_d(\lambda)\,d\lambda,
	\qquad
	I_2(a)\coloneqq\int_1^\infty R_a(\lambda)P_d(\lambda)\,d\lambda.
	\]

	We first show that the tail term is negligible on the scale \(a^{-3}\).
	By \eqref{eq:gamma_ratio}, uniformly in \(a\in[0,1]\) and \(\lambda\ge1\),
	$
	R_a(\lambda)
	\le C \lambda^{-3(d+1)}
	$.	Moreover, again by \eqref{eq:gamma_ratio},
	\[
	P_d(\lambda)\le C\lambda^{d-1},\qquad \lambda\ge1.
	\]
	Therefore
	\[
	R_a(\lambda)P_d(\lambda)\le C\lambda^{-2d-4},\qquad \lambda\ge1.
	\]
	Since \(\lambda\mapsto \lambda^{-2d-4}\) is integrable on \([1,\infty)\), we obtain
	$\sup_{0<a\le1} I_2(a)<\infty$,
	and hence $a^3 I_2(a)\longrightarrow 0$ as $a\downarrow0$.

	It remains to analyse \(I_1(a)\). Substituting \(\lambda=2au\) yields
	\[
	a^3 I_1(a)=\int_0^{1/(2a)} g_a(u)\,du,
	\qquad
	g_a(u)\coloneqq 2a^4R_a(2au)P_d(2au).
	\]
	Fix \(u\ge0\). Since \(\Gamma(z)\sim 1/z\) as \(z\to0\), we have
	\[
	a^6\bigl|\Gamma(a(1+iu))\bigr|^6\longrightarrow (1+u^2)^{-3}, \qquad a\downarrow 0.
	\]
	Furthermore,
	\[
	\left|\Gamma\!\left(a+\frac{d+1}{2}+iau\right)\right|^6
	\longrightarrow
	\Gamma\!\left(\frac{d+1}{2}\right)^6, \qquad a\downarrow 0,
	\]
	and
	\[
	P_d(2au)\sim 4a^2u^2\,\Gamma\!\left(\frac{d-1}{2}\right)^2,\qquad a\downarrow 0.
	\]
	Consequently,
	\[
	g_a(u)\longrightarrow
	8\,\frac{\Gamma\!\left(\frac{d-1}{2}\right)^2}{
		\Gamma\!\left(\frac{d+1}{2}\right)^6}
	\frac{u^2}{(1+u^2)^3},\qquad a\downarrow 0.
	\]

	To justify dominated convergence, note that on the support of \(g_a\) one has
	\(0\le 2au\le1\). Using \(\Gamma(z)=\Gamma(1+z)/z\), we obtain
	\[
	a\,\Gamma(a(1+iu))
	=
	\frac{\Gamma(1+a(1+iu))}{1+iu},
	\]
	and since \(1+a(1+iu)\) ranges in a compact subset of \(\CC\), it follows that
	\[
	a^6|\Gamma(a(1+iu))|^6\le C'(1+u^2)^{-3}
	\]
    for some constant $C'>0$.
	Moreover,
	\[
	\left|\Gamma\!\left(a+\frac{d+1}{2}+iau\right)\right|
	\]
	is bounded away from \(0\) and \(\infty\), uniformly in \(0<a\le1\) and
	\(0\le u\le 1/(2a)\). Finally, the function
	$
	\lambda\mapsto \frac{P_d(\lambda)}{\lambda^2}
	$
	extends continuously to \(0\), hence is bounded on \([0,1]\), so that
	$
	P_d(2au)\le C a^2u^2
	$ for some constant $C>0$.
	Therefore
	\[
	0\le g_a(u)\mathbf 1_{\{u\le1/(2a)\}}
	\le C\frac{u^2}{(1+u^2)^3},
	\qquad u\ge0,
	\]
	and the right-hand side is integrable on \([0,\infty)\). Thus, by the dominated convergence theorem,
	\[
	a^3 I_1(a)
	\longrightarrow
	8\,\frac{\Gamma\!\left(\frac{d-1}{2}\right)^2}{
		\Gamma\!\left(\frac{d+1}{2}\right)^6}
	\int_0^\infty \frac{u^2}{(1+u^2)^3}\,du.
	\]
	Since
	\[
	\int_0^\infty \frac{u^2}{(1+u^2)^3}\,du=\frac{\pi}{16},
	\]
	we obtain
	\[
	a^3 I(a)\longrightarrow
	\frac{\pi}{2}\,
	\frac{\Gamma\!\left(\frac{d-1}{2}\right)^2}{
		\Gamma\!\left(\frac{d+1}{2}\right)^6},\qquad\text{or equivalently}\qquad I(a)\sim
	\frac{\pi}{2}\,
	\frac{\Gamma\!\left(\frac{d-1}{2}\right)^2}{
		\Gamma\!\left(\frac{d+1}{2}\right)^6}\,
	a^{-3}.
	\]
	Since \(K(d,\gamma)\to K(d,\gamma_c^{(d)})\) as \(\gamma\downarrow\gamma_c^{(d)}\), it follows that
	\[
	\EE\cH^d(Z_o)^2
	\sim
	K(d,\gamma_c^{(d)})
	\frac{\pi}{2}\,
	\frac{\Gamma\!\left(\frac{d-1}{2}\right)^2}{
		\Gamma\!\left(\frac{d+1}{2}\right)^6}\,
	a^{-3}.
	\]
	Using
	\[
	a=\frac{c_d}{2}\bigl(\gamma-\gamma_c^{(d)}\bigr),
	\qquad
	a^{-3}=\frac{8}{c_d^3}\bigl(\gamma-\gamma_c^{(d)}\bigr)^{-3},
	\]
	we arrive at
	\[
	\EE\cH^d(Z_o)^2
	\sim
	\left[
	K(d,\gamma_c^{(d)})
	\frac{4\pi}{c_d^3}
	\frac{\Gamma\!\left(\frac{d-1}{2}\right)^2}{
		\Gamma\!\left(\frac{d+1}{2}\right)^6}
	\right]
	\bigl(\gamma-\gamma_c^{(d)}\bigr)^{-3}.
	\]
	The constant simplifies as follows:
	\[
	K(d,\gamma_c^{(d)}) \frac{4\pi}{c_d^3}
	=
	\frac{\pi^{d-2}}{2^{d+3}} \Gamma\!\left(\frac d2\right)^2
	\left(\frac{\gamma_c^{(d)}}{c_d}\right)^3,
	\]
	and
	\[
	\left(\frac{\gamma_c^{(d)}}{c_d}\right)^3
	=
	8\pi^3(d-1)^3
	\frac{\Gamma\!\left(\frac{d+1}{2}\right)^6}{
		\Gamma\!\left(\frac d2\right)^6}.
	\]
	Hence
	\[
	\EE\cH^d(Z_o)^2
	\sim
	\frac{\pi^{d+1}}{2^d}(d-1)^3
	\frac{\Gamma\!\left(\frac{d-1}{2}\right)^2}{
		\Gamma\!\left(\frac d2\right)^4}
	\bigl(\gamma-\gamma_c^{(d)}\bigr)^{-3}.
	\]
	Finally, using
	$
	\Gamma\!\left(\frac{d-1}{2}\right)
	=
	\frac{2}{d-1}\Gamma\!\left(\frac{d+1}{2}\right),
	$
	we obtain
	\[
	\EE\cH^d(Z_o)^2
	\sim
	\frac{\pi^{d+1}}{2^{d-2}}(d-1)
	\frac{\Gamma\!\left(\frac{d+1}{2}\right)^2}{
		\Gamma\!\left(\frac d2\right)^4}
	\bigl(\gamma-\gamma_c^{(d)}\bigr)^{-3},
	\]
	which proves the claim.
\end{proof}

\subsection*{Acknowledgement}

We would like to thank Michael Voit (Dortmund) for useful pointers to the literature and Tobias Hartnick (Karlsruhe) for interesting discussions regarding harmonic analysis in hyperbolic space.\\ The second author was supported by SPP 2265 \textit{Random Geometric Systems}.

\appendix

\end{document}